\documentclass[11pt,a4paper]{amsart}
\usepackage{fullpage}
\usepackage[utf8]{inputenc}
\usepackage[english]{babel}
\usepackage{amsmath}
\usepackage{amsfonts}
\usepackage{amssymb}
\usepackage{amsthm}
\usepackage{amsopn}
\usepackage{mathrsfs}
\usepackage{graphicx}
\usepackage{color}
\usepackage{tikz}
\usepackage{tikz-cd}
\usepackage{extarrows}
\usepackage[arrow, matrix, curve]{xy}
\usepackage{mathtools}
\usepackage{csquotes}
\usetikzlibrary{shapes.misc}

\usepackage[backend=bibtex8]{biblatex}
\newcommand*{\LargerCdot}{\raisebox{-0.25ex}{\scalebox{1.7}{$\cdot$}}}

\theoremstyle{plain}
\newtheorem{thm}{Theorem}[section]
\newtheorem*{theorem*}{Theorem}
\newtheorem{lem}[thm]{Lemma}
\newtheorem{cor}[thm]{Corollary}
\newtheorem{prop}[thm]{Proposition}

\theoremstyle{definition}
\newtheorem{defi}[thm]{Definition}

\newtheoremstyle{remark}{2ex}{2ex}{}{}{\bfseries}{.}{.5em}{}
\theoremstyle{remark}
\newtheorem{rmk}[thm]{Remark}
\newtheorem{nota}[thm]{Notation}

\DeclareMathOperator{\THH}{THH}
\DeclareMathOperator{\HH}{HH}
\DeclareMathOperator{\Tor}{Tor}

\DeclareMathOperator{\colim}{colim}
\DeclareMathOperator{\K}{K}
\DeclareMathOperator{\TC}{TC}
\DeclareMathOperator{\Hom}{Hom}

\DeclareMathOperator{\ku}{ku}

\newcommand{\F}{\mathbb{F}}
\newcommand{\Fp}{\mathbb{F}_p}
\newcommand{\HFp}{{H\mathbb{F}_p}}
\newcommand{\Zp}{\mathbb{Z}_p}
\newcommand{\id}{\mathrm{id}}

\tikzset{commutative diagrams/.cd,
mysymbol/.style={start anchor=center,end anchor=center,draw=none}
}

  \newcommand\MySym[2][(-1)^{l}]{%
  \arrow[mysymbol]{#2}[description]{#1}}
  
   \newcommand\MySymbo[2][(-1)^{t}]{%
  \arrow[mysymbol]{#2}[description]{#1}}

\title{On the Brun spectral sequence for topological Hochschild homology}
\author{Eva H\"oning}
\address{Max-Planck-Institut f\"ur Mathematik,
Vivatsgasse 7, 53111 Bonn, Germany}
\email{hoening@mpim-bonn.mpg.de}

\usepackage{hyperref}
\begin{document}

\begin{abstract}
We generalize a spectral sequence of Brun for the computation of topological Hochschild homology. The generalized version computes the $E$-homology of $\THH(A;B)$, where $E$ is a ring spectrum, $A$ is a commutative $S$-algebra and $B$ is a connective commutative $A$-algebra. 
The input of the spectral sequence are the topological Hochschild homology groups of $B$ with coefficients in the $E$-homology groups of $B \wedge_A B$. 
The mod $p$ and $v_1$  topological Hochschild homology of connective complex $K$-theory has been computed by Ausoni and later again by Rognes, Sagave and Schlichtkrull.   
We present an alternative, short computation using the generalized Brun spectral sequence. 
\end{abstract}

\maketitle
\section{Introduction}

Topological Hochschild homology assists in computations of algebraic $K$-theory.  There is a map $\K(-) \to \THH(-)$, the Dennis trace, from algebraic $K$-theory to topological Hochschild homology that factors over topological cyclic homology.  Topological cyclic homology can be calculated from topological Hochschild homology using homotopy fixed points and Tate spectral sequences. In a number of cases topological cyclic homology is a good approximation of algebraic $K$-theory.  For example, by \cite{HM}  $\K(\mathbb{Z}_p)_p$ is equivalent to the connective cover of $\TC(\mathbb{Z}_p)_p$ and by \cite{DGM} the same follows  for every connective  $S$-algebra $B$ with $\pi_0(B) = \mathbb{Z}_p$.

The goal of this paper is to provide an additional tool for the computation of topological Hochschild homology by generalizing a spectral sequence of Brun \cite{Brun}. 
As an example, we apply the spectral sequence to compute 
the mod $p$ and $v_1$ topological Hochschild homology of $p$-completed connective complex $K$-theory, $V(1)_*\THH(\ku)$, for primes  $p \geq 5$.

The $V(1)$-homotopy of $\THH(\ku)$ has been first  computed by Ausoni in \cite{Au} using the B\"okstedt spectral sequence. 
Later, Rognes, Sagave and Schlichtkrull gave a different proof
using logarithmic topological Hochschild homology \cite{RoSaSch}. 
Starting from his $\THH$-computation Ausoni also computed 
the $V(1)$-homotopy of $\K(\ku)$ \cite{Au2}. 

We use the generalized Brun spectral sequence introduced in this paper in \cite{Ho} to compute the mod $p$ and $v_1$ topological Hochschild homology of the algebraic $K$-theory of finite fields, $V(1)_*\THH(\K(\mathbb{F}_q))$,  in certain cases.

\subsection*{The spectral sequence}

Recall that for a commutative $S$-algebra $B$ the B\"okstedt spectral sequence has the form 
\[ E^2_{*,*} = \HH_{*,*}^{\mathbb{F}_p}\bigl((H\mathbb{F}_p)_*B\bigr) \Longrightarrow (H\mathbb{F}_p)_*\THH(B).\]
Here, $H\mathbb{F}_p$ is the Eilenberg-Mac Lane spectrum of $\mathbb{F}_p$ and $\HH^{\mathbb{F}_p}_{*,*}(-)$ denotes ordinary Hochschild homology over the ground ring $\mathbb{F}_p$. 
The spectral sequence is particularly useful if the pages $E^r_{*,*}$ are flat over the mod $p$ homology of $B$. By \cite{AnRo}  the spectral sequence is an $(H\mathbb{F}_p)_*H\mathbb{F}_p$-comodule $(H\mathbb{F}_p)_*B$-bialgebra spectral sequence in this case.  This structure can be very helpful to compute the differentials.  However, if the flatness condition is not satisfied, computations with the B\"okstedt spectral sequence can be harder. For example, if $B = \ku$ the flatness condition is not satisfied, and  if $B = \K(\mathbb{F}_q)$ it is not always satisfied. 
The B\"okstedt spectral can be generalized from $H\mathbb{F}_p$ to a more general ring spectrum $E$, if we have a 
K\"unneth isomorphism for $E$.  But, for example for the mod $p$ Moore spectrum  $V(0)$ and for 
 $V(1)$ one does not have a K\"unneth isomorphism, so that the B\"okstedt spectral sequence does not allow a direct computation of $V(0)$- or $V(1)$-homotopy. 

The main goal of this paper is to prove the following theorem: 
\begin{thm}
Let $A$ be a cofibrant commutative $S$-algebra and let $B$ be a connective cofibrant commutative $A$-algebra. Let $E$ be a ring spectrum. 
Then, there is a strongly convergent, multiplicative spectral sequence of the  form
\[E^2_{n,m}  = \pi_n\THH(B; HE_m(B \wedge_A B)) \Longrightarrow E_{n+m}\THH(A;B).\] 
If $E_m(B \wedge_A B)$ is an $\mathbb{F}_p$-vector space for all $m$ and if $\pi_0(B)/{p\pi_0(B)} = \mathbb{F}_p$ as rings, we have 
\[E^2_{n,m} = E_m(B \wedge_A B) \otimes_{\mathbb{F}_p} \pi_n\THH(B;H\mathbb{F}_p).\] 
\end{thm}
If  $E = S$ is the sphere spectrum and if $A \to B$ is given by applying the Eilenberg-Mac Lane spectrum functor $H(-)$ to a morphism of commutative rings,  the spectral sequence above has the same form as the one by Brun in   \cite[Theorem 6.2.10]{Brun}. 
Our proof is inspired by Brun's proof, but it is not a direct generalization. 
Brun works with functors with smash product and functors with tensor product and his spectral sequence is based on a bisimplicial abelian group.  
We work with $S$-algebras  in the sense of \cite{EKMM} and our  spectral sequence is an  instance of the Atiyah-Hirzebruch spectral sequence.

\subsection*{Application to ku}
We  apply the generalized Brun spectral sequence to ku. The computation is organized in the following three steps:  
\begin{itemize}
\item Computation of 
\[V(0)_*(H\mathbb{Z}_p \wedge_{\ku} H\mathbb{Z}_p) = \pi_*(H\mathbb{F}_p \wedge_{\ku} H\mathbb{Z}_p)\]
via the K\"unneth spectral sequence.
Here, $\mathbb{Z}_p$ is the ring of $p$-adic integers.
\item  Computation of $V(0)_*\THH(\ku; H\mathbb{Z}_p)$ via the generalized Brun spectral sequence 
\[ E^2_{*,*} = V(0)_*(H\mathbb{Z}_p \wedge_{\ku} H\mathbb{Z}_p) \otimes_{\mathbb{F}_p} \pi_*\THH(H\mathbb{Z}_p; H\mathbb{F}_p) \Longrightarrow V(0)_*\THH(\ku; H\mathbb{Z}_p).\]
Here, note that $\pi_*\THH(H\mathbb{Z}_p; H\mathbb{F}_p)$ has been computed by B\"okstedt \cite{Bo2}. 
\item Computation of $V(1)_*\THH(\ku)$ via the generalized Brun spectral sequence 
\[ E^2_{*,*} = V(1)_*\ku \otimes_{\mathbb{F}_p} \pi_*\THH(\ku; H\mathbb{F}_p) \Longrightarrow V(1)_*\THH(\ku).\]
Here, note that $\pi_*\THH(\ku; H\mathbb{F}_p) = V(0)_*\THH(\ku; H\mathbb{Z}_p)$. 
\end{itemize} 
The first generalized Brun spectral sequence mentioned above has to collapse at the $E^2$-page and there cannot be any multiplicative extensions. 
The differentials in the second generalized Brun spectral sequence mentionend above can be computed by using the relation $u^{p-1} = 0$ in $(H\mathbb{F}_p)_*\ku$, where $u \in (H\mathbb{F}_p)_2\ku$ is the image of the Bott element under the Hurewicz map $\pi_*(\ku) \to (H\mathbb{F}_p)_*\ku$. 
 In contrast to the Bökstedt spectral sequence  for $\ku$ there are no multiplicative extensions in the Brun spectral sequence  and we obtain the following concise description of $V(1)_*\THH(\ku)$:
 
\begin{thm}
The $V(1)$-homotopy of $\THH(\ku)$ is the homology of the differential graded algebra 
\[ P_{p-1}(u) \otimes_{\mathbb{F}_p} E(\sigma u, \lambda_1) \otimes_{\mathbb{F}_p} P(\mu_1), \, \, \, d(\mu_1) = u^{p-2} \sigma u.\] 
Here, $P(-)$ denotes the polynomial algebra over $\mathbb{F}_p$, $E(-)$ denotes the exterior algebra over $\mathbb{F}_p$ and $P_{p-1}(-)$ denotes the truncated polynomial algebra over $\mathbb{F}_p$.  The degrees are given by $|u| = 2$, $|\sigma u | = 3$, $|\lambda_1| = 2p-1$ and $|\mu_1| = 2p$. 
\end{thm}

\subsection*{Structure of the paper}

We  work in the setting of Elmendorf, Kriz, Mandell and May \cite{EKMM}. In Section \ref{Presec} we collect some of the properties of the categories we are working with. 
The main goal of Section \ref{AtSec} is to give a detailed proof of the multiplicativity of the Atiyah-Hirzebruch spectral sequence. 
In Subsection \ref{Multsec} we transfer some methods of   \cite{Du} by Dugger about multiplicative structures on spectral sequences to the EKMM-setting.  
In Subsection  \ref{AHsec} we use this to prove the multiplicativity of the Atiyah-Hirzebruch spectral sequence and some additional properties. 
The content of Sections \ref{Presec} and \ref{AtSec} consist essentially of recollections, and are included here to provide details that we did not find in the literature.  
In Section \ref{Brunsubsec} we apply the Atiyah-Hirzebruch spectral sequence to derive the existence of the generalized Brun spectral sequence. 
Finally, in Section  \ref{kusec} we compute $V(1)_*\THH(\ku)$ using the Brun spectral sequence. 

\subsection*{Acknowledgments}
The content of this article is part of my PhD thesis. I would like to thank my PhD supervisor Christian Ausoni  for his help. 
This work was supported by grants from R\'egion \^Ile-de-France and the project ANR-16-CE40-0003 ChroK.  I also would like to thank the MPIM Bonn for providing ideal working conditions during my postdoc position.

\section{Preliminaries} \label{Presec}

We assume familiarity with the theory of Elmendorf, Kriz, Mandell and May \cite{EKMM}. 
 We denote the category of  (compactly generated weak Hausdorff) spaces by $\mathscr{U}$ and the category of based spaces by $\mathscr{T}$.   
Let $\mathscr{S}$, $\mathscr{S}[\mathbb{L}]$ and $\mathscr{M}_S$ be the categories of spectra, of $\mathbb{L}$-spectra and of  $S$-modules,  respectively. 
For a commutative $S$-algebra $R$  let $\mathscr{M}_R$ be the category of $R$-modules. 
We denote the tensor products of the closed symmetric monoidal categories  $\mathscr{M}_S$ and $\mathscr{M}_R$ by $\wedge_S$ and $\wedge_R$.

Recall from \cite[Chapter VII]{EKMM} that the categories $\mathscr{S}$, $\mathscr{S}[\mathbb{L}]$, $\mathscr{M}_S$ and $\mathscr{M}_R$ are  topological model categories.  In particular, these categories are topologically complete and cocomplete. The tensor  with a based space $X$ is denoted by $- \wedge X$. Colimits and tensors in $\mathscr{M}_R$, $\mathscr{M}_S$ and $\mathscr{S}[\mathbb{L}]$ are created in $\mathscr{S}$. 
By \cite[p.132]{EKMM} the functors considered in \cite{EKMM} are enriched functors.
We point out  that 
for a $\mathscr{T}$-enriched functor $F: \mathscr{C} \to \mathscr{D}$ between $\mathscr{T}$-tensored categories we have a natural morphism 
 \begin{equation*} \label{reinz}
 \begin{tikzcd}
 F(c) \wedge X \ar{r} & F(c \wedge X), 
 \end{tikzcd}
 \end{equation*}
 which is an isomorphism if  $F$ is the left adjoint of an adjoint pair of $\mathscr{T}$-enriched functors.
 A homotopy  is a  map of the form $E \wedge I_+ \to F$ in the respective category, where $I = [0,1]$ is the unit interval, and the homotopy categories are denoted by   $h\mathscr{S}$, $h\mathscr{S}[\mathbb{L}]$, $h\mathscr{M}_S$ and $h\mathscr{M}_R$.  




We denote the homotopy category of the model category $\mathscr{M}_R$ by $\mathscr{D}_R$. 
Since all objects in $\mathscr{M}_R$ are fibrant, cell $R$-modules are cofibrant and since $M \wedge I_+$ is a good cylinder object if $M$ is a cofibrant $R$-module, we get that 
\[ \mathscr{D}_R(M, N) \cong h\mathscr{M}_R(\Gamma^R M, \Gamma^RN) \cong h\mathscr{M}_R(\Gamma^RM, N),\]  where $\Gamma^R M$ and $\Gamma^RN$ are cell approximations of $M$ and $N$. 
The category $\mathscr{D}_R$ is a tensor triangulated category in the sense of \cite[Definition 1.1]{Bal}. 
 The tensor product 
$\wedge_R^L$  is given by
\[ M \wedge_R^L N \coloneqq \Gamma^R M \wedge_R \Gamma^R N.\]
If $M$ is a cofibrant $R$-module, then $M \wedge_R -$ preserves weak equivalences. Thus,  for any $M, N \in \mathscr{D}_R$  we have that $M \wedge_R^L N$ is isomorphic in $\mathscr{D}_R$ to $\Gamma^RM \wedge_R N$ and $M \wedge_R \Gamma^R N$.

To deal with the signs in the spectral sequences we follow Dugger's treatment in \cite{Du} and fix orientations on various  manifolds:
\begin{rmk}  \label{or}
 Equip $\mathbb{R}$ with the usual orientation and $\mathbb{R}^n$ with the product orientation. 
These orientations induce orientations on $D^n$ 
and $I$.  
We equip $I$ with the basepoint $0$, and $D^n$  and the sphere $S^{n-1}$ with the basepoint $(-1, 0, \dots, 0)$. 
The orientations on $D^n$ and $I$ induce orientations on the boundaries $S^{n-1}$ and $\{0,1\}$. Here, we use  the same convention as in  \cite[Remark 2.2]{Du}  for the boundary orientation, so that the basepoint in $S^0$ has orientation $-1$ and  the non-basepoint in $S^0$ has orientation $+1$ and so that we have the  
following formula for two manifolds $M$, $N$ with boundary:
\begin{equation} \label{randformel}
\partial(M \times N) = \partial M \times N \cup (-1)^{\dim(M)} M \times \partial N.
\end{equation}
 The orientations on the interiors of $D^n$  and $I$  determine orientations on $D^n/S^{n-1}$  and  on $I/{\partial I}$.  
For $n \geq 1$ or $m \geq 1$ the product orientation on the right-hand side of 
\[S^n \wedge S^m  \setminus \{*\} \cong S^{n} \setminus \{\ast\} \times S^m \setminus \{\ast\}\]
 defines an orientation on $S^n \wedge S^m$.  In $S^0 \wedge S^0$ we equip the non-basepoint with the orientation $1$ and the basepoint with the orientation $-1$. 
For the rest of this paper we fix basepoint-preserving, orientation-preserving homeomorphisms $D^n/S^{n-1} \cong S^n$, $S^n \wedge S^m
 \cong S^{n+m}$  and $I/{\partial I} \cong S^1$. 
\end{rmk}

To deal with the signs in the spectral sequences we give a precise definition of suspension isomorphisms and boundary maps in long exact sequences: 
Recall that we have left adjoint functors 
 \[\begin{tikzcd}[column sep = 8 ex]
\mathscr{S}    \ar{r}{\mathbb{L}(-)}  &  \mathscr{S}[\mathbb{L}]    \ar{r}{S \wedge_{\mathscr{L}} (-)} &  \mathscr{M}_S \ar{r}{R \wedge_S (-)} &  \mathscr{M}_R,
 \end{tikzcd}\] 
  whose composition is denoted by  $\mathbb{F}_R$. 
Spheres in $\mathscr{M}_R$  are denoted by $S_R^n$ and  are defined by applying $\mathbb{F}_R(-)$ to the  spheres $S^n$ in $\mathscr{S}$.  Note  that the $S^n_R$ are CW $R$-modules and that $S^0_R$ is a cell approximation of $R$. 
The homotopy groups of an $R$-module $M$  are given by 
\[ \pi_n(M) = h\mathscr{M}_R(S^n_R, M). \]
For $n \geq 0$ we define an isomorphism $S^n \wedge S^1 \cong S^{n+1}$  in $\mathscr{S}$ by 
  \[ \Sigma^{\infty}S^{n} \wedge S^1 \cong \Sigma^{\infty} (S^{n} \wedge S^1) \cong \Sigma^{\infty}S^{n+1}.\]
  Here, $\Sigma^{\infty}$ is the left adjoint to the zeroth space functor $\Omega^{\infty}$ and 
the second map is given by the fixed orientation-preserving homeomorphism.  For $n < 0$ we define an isomorphism $S^n \wedge S^1 \cong S^{n+1}$ using 
 \cite[Proposition I.4.2]{LMSt}. 
 We define isomorphisms $S^n_R \wedge S^1 \cong S^{n+1}_R$  by 
 \[ \mathbb{F}_RS^n \wedge S^1 \cong \mathbb{F}_R(S^{n} \wedge S^1) \cong \mathbb{F}_RS^{n+1}.\]
 These isomorphisms define  suspension isomorphisms $\pi_n(M) \cong \pi_{n+1}(M \wedge S^1)$ for $M \in \mathscr{M}_R$. 
 Let  $f: M \to N$  be a morphism in  $\mathscr{M}_R$ and let $C(f) \coloneqq N \cup_M  (M \wedge I)$  be its cofiber.  We use the fixed homeomorphism $I/{\partial I} \cong S^1$ and the suspension isomorphism to define the boundary map in the long exact  sequence
\[\begin{tikzcd}
 \cdots \ar{r} & \pi_n(M) \ar{r} & \pi_n(N) \ar{r} &  \pi_n(C(f)) \ar{r}{d} &  \pi_{n-1}(M) \ar{r} & \cdots.
 \end{tikzcd}\] 
If $f$ satisfies the homotopy extension property, then the canonical map $C(f) \to N/M \coloneqq \ast \cup_M N$ is a homotopy equivalence.
Following \cite{EKMM} we call a map in $\mathscr{M}_R$ satisfying the homotopy extension property a cofibration.
 Note that a cofibration is not necessary a cofibration in the model categorical sense.

In the proof of the existence of the Brun spectral sequence we will need that different functors behave  
well with respect to algebras and ring spectra. For this observe the following: 
The canonical functor $i_R: \mathscr{M}_R \to \mathscr{D}_R$ is lax symmetric monoidal. The structure maps are given by applying $i_R$ to the maps
\[\begin{tikzcd}
 \Gamma^R M \wedge_R \Gamma^R N \ar{r} & M \wedge_R N. 
\end{tikzcd}\]
Moreover, for a morphism of commutative $S$-algebras $\phi: R \to R'$ we have a commutative diagram 
\[\begin{tikzcd}
    \mathscr{M}_{R'} \ar{r}{\phi^*} \ar{d}{i_{R'}} & \mathscr{M}_R \ar{d}{i_R} \\
    \mathscr{D}_{R'} \ar{r}{\phi^*} & \mathscr{D}_R.
    \end{tikzcd}\]
All the functors in this diagram are lax symmetric monoidal and the structure maps of the two lax monoidal functors $i_R \circ \phi^*$  and $\phi^* \circ i_{R'}$ agree.  
Now, let $\phi \colon R \to R'$ and $\psi \colon R' \to  R''$ be maps of commutative $S$-algebras. Then, the two functors 
  $ (\psi \circ \phi)^*, \phi^* \circ \psi^*: \mathscr{D}_{R''} \to \mathscr{D}_{R}$ are the same and the structure maps of these  two lax monoidal functors agree.
 Recall that an $R$-ring spectrum is an object $A \in \mathscr{D}_R$ with maps 
$A \wedge_R^LA \to A$ and  $R \to A$ in $\mathscr{D}_R$ satisfying the left and right unit laws. 
As  a consequence of the above we get that for $R$-algebras $M$ and $N$ the map
\[\begin{tikzcd}
M \wedge_R^L N \ar{r} & M \wedge_R N 
\end{tikzcd}\]
in $\mathscr{D}_R$ is a map of $R$-ring spectra. 
Moreover, for maps of commutative $S$-algebras $R \xrightarrow{\phi} R' \xrightarrow{\psi} R''$ we get commutative diagrams
 \[\begin{tikzcd}
    \mathscr{A}_{R'} \ar{r}  \ar{d} & \mathscr{A}_{R} \ar{d} \\
    \mathscr{R}_{R'} \ar{r} & \mathscr{R}_{R}
    \end{tikzcd}\] 
    and 
    \[\begin{tikzcd}
    \mathscr{R}_{R''} \ar{r}{\psi^*} \ar{rd}[swap]{(\psi \circ \phi)^*} & \mathscr{R}_{R'} \ar{d}{\phi^*} \\
    & \mathscr{R}_{R}.
    \end{tikzcd}\]
  Here,  $\mathscr{A}_R$ and $\mathscr{A}_{R'}$ denote the categories of algebras over $R$ and $R'$, and $\mathscr{R}_R$, $\mathscr{R}_{R'}$ and $\mathscr{R}_{R''}$ denote the categories of ring spectra over $R$, $R'$ and $R''$, respectively.

 We now recall some details about the lax monoidality of the functor $\pi_*(-)$, because we need these in the proof of the multiplicativity of the spectral sequences. 
 By \cite[p.92]{EKMM} there are homotopy equivalences  
\begin{equation*} \label{speq}
S^k_R \wedge_R S^l_R \simeq S^{k+l}_R,
\end{equation*}
  that are unital, associative and commutative up to the sign $(-1)^{kl}$.  They define 
  maps
\begin{equation*}
\begin{tikzcd} \label{Kunnethmor}
\mathscr{D}_R(S^k_R, M) \otimes_{\mathbb{Z}} \mathscr{D}_R(S^l_R, N) \ar{r} & \mathscr{D}_R(S^{k+l}_R, M \wedge_R^L N)
\end{tikzcd}
\end{equation*}
for $M, N \in \mathscr{D}_R$ which make $\pi_*(-)$ into a lax symmetric monoidal functor from $\mathscr{D}_R$ to the category of $\mathbb{Z}$-graded abelian groups $\mathscr{G}\mathscr{A}$. Here, the symmetry in $\mathscr{G}\mathscr{A}$  is given by 
\begin{equation*} \label{signsym}
 A_* \otimes_{\mathbb{Z}} B_* \to B_* \otimes_{\mathbb{Z}} A_*; \, \, \, a \otimes b \mapsto (-1)^{|a||b|} b \otimes a, 
 \end{equation*} 
 where $|a|$ and $|b|$ denotes the degrees of $a$ and $b$. 
The equivalences $S^k_R \wedge_R S^l_R \simeq S^{k+l}_R$ are compatible with  suspension in the sense that the diagrams
\begin{equation} \label{susp1}
\begin{tikzcd} 
S^k_R \wedge_R S^l_R \ar{rr} \MySym{drr} \ar{d} & & S^{k+l}_R \ar{d} \\
(S^{k-1}_R \wedge S^1) \wedge_R S^l_R \ar{r} & (S^{k-1}_R \wedge_R S^l_R) \wedge S^1 \ar{r} & S^{k+l-1}_R \wedge S^1   
\end{tikzcd}
\end{equation}

\begin{equation} \label{susp2}
\begin{tikzcd}
S^k_R \wedge_R S^l_R \ar{rr} \ar{d} & & S^{k+l}_R \ar{d} \\
S^k_R \wedge_R (S^{l-1}_R \wedge S^1) \ar{r} & (S^k_R \wedge_R S^{l-1}_R) \wedge S^1 \ar{r} &  S^{k+l-1}_R  \wedge S^1  
\end{tikzcd}
\end{equation}
commute in $\mathscr{D}_R$ up to the indicated sign. 
 We will later compare  the product on usual homotopy groups to a product on relative homotopy groups.  
For this we need that for $k, l \geq 0$ the equivalence $S^k_R \wedge_R S^l_R \simeq S^{k+l}_R$ can be identified with the following map: 
By  \cite[Corollary III.3.7]{EKMM}, \cite[Proposition II.3.6]{LMSt} and  \cite[Proposition II.1.4]{LMSt} there is a natural isomorphism 
\begin{equation}\label{smiso}
\mathbb{F}_R\Sigma^{\infty}X \wedge_R \mathbb{F}_R\Sigma^{\infty}Y \cong \mathbb{F}_R\Sigma^{\infty}(X \wedge Y) 
\end{equation} 
for $X, Y \in \mathscr{T}$. 
 For $k,l \geq 0$ the  equivalence $S^k_R \wedge_R S^l_R \simeq S^{k+l}_R$ is then homotopic to  the map that is given by taking $X = S^k$ and $Y = S^l$ in (\ref{smiso}) and by using the orientation-preserving homeomorphism $S^k \wedge S^l \cong S^{k+l}$.  
Let $R \to R'$ be a morphism of commutative $S$-algebras. 
 Then, for $M \in \mathscr{D}_{R'}$ one has  natural  isomorphisms
\begin{equation*} \label{zwHom}
\mathscr{D}_{R'}(S^n_{R'}, M) \cong \mathscr{D}_R(S^n_R, M)
\end{equation*}
compatible with the suspension isomorphisms.
  One can show that they define a monoidal natural isomorphism between the functors 
\[\pi_*: \mathscr{D}_{R'} \to \mathscr{G}\mathscr{A} \text{~~~~and~~~} \mathscr{D}_{R'} \to \mathscr{D}_{R} \xrightarrow{\pi_*} \mathscr{GA}.\]

Let us denote the category of commutative $S$-algebras by
$\mathscr{C}\mathscr{A}_S$.  Note that the pushout of a diagram $G \leftarrow H \rightarrow J$ in $\mathscr{C}\mathscr{A}_S$ is given by  $G \wedge_H J$. 
By \cite[Chapter VII]{EKMM} the  category  $\mathscr{C}\mathscr{A}_S$  is a $\mathscr{U}$-enriched topological model category.  
We denote the homotopy category of this model category  by $\bar{h}\mathscr{C}\mathscr{A}_S$. The tensor of a commutative $S$-algebra $A$ and a space $X$ is denoted by $A \otimes X$. 
If $R$ is a cofibrant commutative $S$-algebra, the tensor   $R \otimes I$ is a good cylinder object. Since all objects are fibrant, it follows that for commutative $S$-algebras $R$ and $R'$ 
the set $\bar{h}\mathscr{C}\mathscr{A}_S(R,R')$ is the set of homotopy classes
of maps of commutative $S$-algebras $QR \to QR'$, where $QR$ and $QR'$ are cofibrant replacements of $R$ and $R'$  in $\mathscr{C}\mathscr{A}_S$  and  where
homotopies are given  by maps $QR \otimes I \to QR'$. One has a map $R \wedge I_+ \to R  \otimes I$ for every commutative $S$-algebra $R$, so that  maps that are homotopic as maps of commutative $S$-algebras  are homotopic as  maps of $S$-modules.

Let $R$ be a commutative $S$-algebra. We denote the category of commutative $R$-algebras by $\mathscr{C}\mathscr{A}_R$.  As $\mathscr{C}\mathscr{A}_R$ can be identified with the category of commutative $S$-algebras under $R$, it inherits a model category structure from $\mathscr{C}\mathscr{A}_S$.  This model category structure is the same as the one defined in \cite[Chapter VII]{EKMM}, because in both cases weak equivalences are created in $\mathscr{S}$ and fibrations are created in $\mathscr{M}_S$.

 \begin{lem} \label{cof1en}
 Let $A$ be a cofibrant commutative $S$-algebra, let $A \to B$ be a cofibration in $\mathscr{C}\mathscr{A}_S$ and let $A \to C$ be a map in $\mathscr{C}\mathscr{A}_S$, where  $C$ is a cofibrant commutative $S$-algebra. Then the map 
\[ B \wedge_A^L C \to B \wedge_A C\]
in $\mathscr{D}_A$ is an isomorphism of $A$-ring spectra. 
 \end{lem}
 \begin{proof}
 We only have to show that the map  $\Gamma B \wedge_ A \Gamma C \to B \wedge_A C$ is a weak equivalence.  
 To see this, factor the map $A \to C$ in $\mathscr{C}\mathscr{A}_S$ as 
 \[\begin{tikzcd}
 A \ar[tail]{r} & \tilde{C} \ar{r}{\sim} & C  
 \end{tikzcd}\]
 and consider the homotopy commutative diagram
 \[\begin{tikzcd}
 \Gamma B \wedge_A \Gamma C \ar{r} & B \wedge_A C \\
 \Gamma B \wedge_A \Gamma \tilde{C}  \ar{u} \ar{r} & B \wedge_A \tilde{C} \ar{u}. 
 \end{tikzcd}\]
 Obviously, the left vertical map is a weak equivalence.  The lower horizontal map is a weak equivalence by \cite[Theorem VII.6.5, Theorem VII.6.7]{EKMM} and the right vertical map is a weak equivalence by  \cite[Theorem VII.7.4]{EKMM}
 \end{proof}

\section{The multiplicativity of the Atiyah-Hirzebruch spectral sequence} \label{AtSec}
 
\subsection{Multiplicative spectral sequences associated to towers of cofibrations} \label{Multsec}
The goal of this subsection is to transfer some of the methods of \cite{Du} to the EKMM-setting. 
We fix a commutative $S$-algebra $R$. 
  
\begin{lem} \label{ss}
Let $\cdots \to X_{n-1} \to X_n \to X_{n+1} \to \cdots$ be a sequence of cofibrations of $R$-modules. 
Then there  is a spectral sequence $E^{*}_{*,*}(X_{\LargerCdot})$ of standard homological type with 
\begin{equation} \label{ssbase}
E^1_{n,m} = \pi_{n+m}(X_n/X_{n-1}).
\end{equation} 
If $X_n = \ast$ for $n < 0$, then the spectral sequence converges strongly: 
\[E^1_{n,m} \Longrightarrow \pi_{n+m}(\colim_i X_i).\]
\end{lem}
\begin{proof}
We have the following  unrolled exact couple:
\begin{equation*} \label{uecbase}
\begin{tikzcd}
 \cdots \ar{r}{i} & \pi_*(X_{n-1}) \ar{r}{i} \ar{d}{j} & \pi_*(X_n) \ar{d}{j} \ar{r}{i} & \pi_*(X_{n+1}) \ar{d}{j} \ar{r}{i} & \cdots \\
   &  \pi_*(X_{n-1}/X_{n-2}) &  \pi_*(X_n/X_{n-1})  \ar{lu}{\partial} & \pi_*(X_{n+1}/X_n) \ar{lu}{\partial} &  
\end{tikzcd}
\end{equation*}
Here, $\partial$ is defined by $(-1)^m$ times the map 
\[ \pi_m(X_n/X_{n-1}) \cong \pi_{m}(C(X_{n-1} \to X_n)) \xrightarrow{d} \pi_{m-1}(X_{n-1}) .\] 
The reason for the sign will become clear in Lemma \ref{compLES}. This defines the spectral sequence. 
Clearly, if $X_n = \ast$ for $n < 0$, one has $\pi_*(X_n/X_{n-1}) = 0$ for $n < 0$ and $\lim_n \pi_*(X_n)= 0$. So, by \cite[Theorem 6.1]{Boar}  
the spectral sequence converges strongly to $\colim_n \pi_*(X_n)$.

We show that the canonical map 
\[\alpha: \colim_n \pi_m(X_n) \to \pi_m(\colim_n X_n)\]
 is a bijection: Each map $X_{n-1} \to X_n$ is a cofibration of $R$-modules and thus a cofibration of spectra by the retraction of mapping cylinders criterion. By \cite[Lemma I.8.1]{LMSt} each map $X_{n-1} \to X_n$ is a spacewise closed inclusion.
 Thus, by \cite[Proposition III.1.7]{EKMM}, every map $L \to \colim_n X_n$, where $L$ is a compact $R$-module in the sense of
   \cite[Definition III.1.6]{EKMM}, factors through some $X_i$. 
Applying this for $L = S^m_R$ and $L = S^m_R \wedge I_+$ one gets  the surjectivity and injectivity of $\alpha$. 
For the injectivity also  note that  the maps $X_j \to \colim_n X_n$ are cofibrations and thus  monomorphisms in the category of $R$-modules. 
\end{proof}

To deal with multiplicative structures, we follow Dugger's treatment and work with homotopy groups of morphisms. 
For a map $M \to N$ of $R$-modules and $n \geq 1$ we define  $\pi_n(N,M)$ as the set of homotopy classes of diagrams
\begin{equation}  \label{reldiag}
\begin{tikzcd}
\mathbb{F}_R\Sigma^{\infty} S^{n-1} \ar{r}{\beta} \ar{d} & M \ar{d} \\
\mathbb{F}_R \Sigma^{\infty} D^n \ar{r}{\alpha} & N.
\end{tikzcd}
\end{equation}
Recall that the right adjoint of  $\mathbb{F}_R\Sigma^{\infty}(-)$ is given by  the  composition
\[\begin{tikzcd}
\mathscr{M}_R \ar{r} & \mathscr{M}_S \ar{r}{F_{\mathscr{L}}(S,-)} & \mathscr{S}[\mathbb{L}] \ar{r} & \mathscr{S} \ar{r}{\Omega^{\infty}} & \mathscr{T}, 
\end{tikzcd}\]
 where $F_{\mathscr{L}}(-,-)$ is the function $\mathbb{L}$-spectrum. Thus,  we have  a bijection 
\[\pi_n(N,M) \cong \pi_n\bigl(\Omega^{\infty}F_{\mathscr{L}}(S,N),\Omega^{\infty}F_{\mathscr{L}}(S,M)\bigr),\] 
where the right-hand side is the homotopy group of a morphism of based spaces defined in \cite{Du}.  Using \cite{Du} one gets that 
 $\pi_n(N,M)$ has a natural group structure for $n = 2 $ and  a natural abelian group structure for $n \geq 3$,  that $\pi_n(M, *) \cong \pi_n(M)$ and that 
the  sequence 
\begin{equation*}\label{exsrel}
\begin{tikzcd}
\pi_n(M) \ar{r} & \pi_n(N) \ar{r} & \pi_n(N,M) \ar{r}{\kappa} & \pi_{n-1}(M) \ar{r} &  \pi_{n-1}(N)
\end{tikzcd}
\end{equation*}
is exact. 
Here, 
 $\kappa$ is the map that sends the class of the diagram  (\ref{reldiag}) to the class of $\beta$.

\begin{lem} \label{compLES}
Let $M \to N$ be a cofibration of $R$-modules.  Then,
 the diagram 
\[\begin{tikzcd}
\pi_n(M) \ar{r} \ar{d}{\id} & \pi_n(N) \ar{r} \ar{d}{\id} & \pi_n(N,M) \ar{r}{\kappa} \ar{d}{\phi} & \pi_{n-1}(M) \ar{r} \ar{d}{\id} & \pi_{n-1}(M) \ar{d}{\id} \\
\pi_n(M) \ar{r}  & \pi_n(N) \ar{r}  & \pi_n(N/M) \ar{r}{(-1)^n d} & \pi_{n-1}(M) \ar{r} & \pi_{n-1}(M). 
\end{tikzcd}\]
is commutative  and $\phi$ is an isomorphism. 
\end{lem}
\begin{proof}
For the commutativity of the third square note that by naturality we only have to consider  the diagram 
\[\begin{tikzcd}
\mathbb{F}_R \Sigma^{\infty} S^{n-1} \ar{r}{\id} \ar{d} & \mathbb{F}_R \Sigma^{\infty} S^{n-1} \ar{d} \\
\mathbb{F}_R \Sigma^{\infty} D^n \ar{r}{\id} & \mathbb{F}_R \Sigma^{\infty} D^n.
\end{tikzcd}  \]
For the commutativity on this element one has to show that 
\[\begin{tikzcd}
C(\mathbb{F}_R\Sigma^{\infty}S^{n-1} \to \mathbb{F}_R \Sigma^{\infty} D^n) \ar{r} \ar{d}[swap]{\simeq} &   \mathbb{F}_R\Sigma^{\infty} S^{n-1} \wedge S^1 \ar{d} \\
\mathbb{F}_R D^n/{\mathbb{F}_R \Sigma^{\infty}S^{n-1}} \cong \mathbb{F}_R\Sigma^{\infty}D^n/{S^{n-1}} \ar{d}[swap]{\cong} & \mathbb{F}_R \Sigma^{\infty} (S^{n-1} \wedge S^1) \ar{d} \\
\mathbb{F}_R\Sigma^{\infty} S^n \ar{r}{(-1)^n \id} & \mathbb{F}_R \Sigma^{\infty}S^n 
\end{tikzcd}
\]
is commutative in the homotopy category of $R$-modules. Roughly speaking, this  follows by applying the formula (\ref{randformel})
for the boundary orientation of a product to $D^n \times I$. 
 By the five lemma $\phi$ is an isomorphism for $n \geq 2$.  For $n = 1$ note that we can endow $\pi_1(N,M)$ with an abelian group structure by identifying it with the set of 
homotopy classes of commutative diagrams
\[\begin{tikzcd}
S^{-2}_R \wedge S^2 \ar{d} \ar{r}  & M \ar{d} & \\
(S^{-2}_R \wedge I) \wedge S^2 \ar{r} & N
\end{tikzcd}\]
and by using the comultiplication of $S^2$. 
One easily sees that $\pi_1(N) \to \pi_1(N,M)$, $\kappa$ and $\phi$ are group homomorphism with respect to this group structure. 
\end{proof}


Following Dugger we now consider products on $\pi_*(-,-)$:  
Let $D^{n+m} \to D^n \times D^m$ be  an orientation-preserving, pointed homeomorphism. It induces an orientation-preserving homeomorphism of the boundaries 
\[S^{n+m-1}  \xrightarrow{\cong} (S^{n-1} \times D^m)\, \amalg_{(S^{n-1} \times S^{m-1})} \, (D^n \times S^{m-1})\]
 and we get the following commutative diagram:
\begin{equation} \label{pr1}
\begin{tikzcd}
S^{n+m-1} \ar{r} \ar{d} & (S^{n-1} \wedge D^m)\, \amalg_{(S^{n-1} \wedge S^{m-1})} \, (D^n \wedge S^{m-1}) \ar{d}  \\
D^{n+m-1} \ar{r} & D^n \wedge D^m.
\end{tikzcd}
\end{equation}
Given two maps  of $R$-modules $M \to N$ and $O \to L$
 we define  a product 
\[(-) \cdot (-): \pi_n(N,M) \times \pi_m(L, O) \to \pi_{n+m}\bigl(N \wedge_R L, (M \wedge_R L) \, \amalg_{(M \wedge_R O)} \, (N \wedge_R O) \bigr)\]
by applying $\mathbb{F}_R\Sigma^{\infty}(-)$ to (\ref{pr1}) and by using the isomorphism (\ref{smiso}).

\begin{lem} \label{prodcomp2}
Let $M \to N$ and $O \to L$ be maps of $R$-modules. Then the following diagram commutes: 
\begin{equation*} \label{compprodmod}
\begin{tikzcd} 
\pi_k(N,M) \times \pi_l(L,O) \ar{d} \ar{r} & \pi_{k+l}\bigl(N \wedge_R L, (M \wedge_R L) \, \amalg_{(M \wedge_R O)} \,(N \wedge_R O)\bigr) \ar{d} \\
\pi_k(N/M) \times \pi_l(L/O) \ar{rd} & \pi_{k+l}\bigl(N \wedge_R L/{(M \wedge_R L \, \amalg_{M \wedge_R O} \,N \wedge_R O)}\bigr)\ar{d} \\ 
& \pi_{k+l}(N/M \wedge_R L/O). 
\end{tikzcd}
\end{equation*}
\end{lem}
\begin{proof}
Because of naturality  we only have to prove commutativity in
the universal case. 
 This can be reduced to the corresponding statement for  spaces by using that we 
 have natural maps 
 \begin{equation} \label{spspcomp1}
 \pi_*(B, A) \to  \pi_*(\mathbb{F}_R\Sigma^{\infty}B,\mathbb{F}_R\Sigma^{\infty}A) 
 \end{equation}
 and 
 \begin{equation} \label{spspcomp2}
    \pi_*(B) \to \pi_*(\mathbb{F}_R\Sigma^{\infty}B) 
    \end{equation}
that identify for $A = \ast$ and that are compatible with products.  Here, the product on the left-hand side of (\ref{spspcomp2}) is defined by means of the fixed orientation-preserving homeomorphisms $S^k \wedge S^l \cong S^{k+l}$. 
 The statement for spaces follows from the fact that two pointed, orientation-preserving homeomorphisms $S^{k+l} \to D^k/{S^{k-1}} \wedge D^l/{S^{l-1}}$ are homotopic. 
\end{proof}
\begin{lem} \label{Leibniz}
Let $M \to N$ and $O \to L$ be morphisms of $R$-modules. 
Let $j_*$, $(i_1)_*$ and $(i_2)_*$ be the maps 
\begin{eqnarray*}
\pi_*\bigl((M \wedge_R L) \, \amalg_{(M \wedge_R O)} \, (N \wedge_R O) \bigr) & \to & \pi_{*}\bigl((M \wedge_R L) \, \amalg_{(M \wedge_R O)} \,(N \wedge_R O), M \wedge_R O\bigr) \\
\pi_*\bigl(M \wedge_R L, M \wedge_R O \bigr) & \to & \pi_{*}\bigl((M \wedge_R L) \, \amalg_{(M \wedge_R O)} \,(N \wedge_R O), M \wedge_R O\bigr) \\
\pi_*\bigl( N \wedge_R O, M \wedge_R O \bigr) & \to & \pi_{*}\bigl((M \wedge_R L) \, \amalg_{(M \wedge_R O)} \,(N \wedge_R O), M \wedge_R O\bigr)
\end{eqnarray*}
respectively. Let 
$x \in \pi_k(N,M)$ and $y \in \pi_l(L,O)$ with $k,l \geq 3$.
Then, we have
\[j_*\kappa(x \cdot y) = (i_1)_*(\kappa x \cdot y) + (-1)^k (i_2)_*(x \cdot \kappa y) \]
in $\pi_{k+l-1}\bigl((M \wedge_R L) \, \amalg_{(M \wedge_R O)} \,(N \wedge_R O), M \wedge_R O\bigr)$.  
\end{lem}
\begin{proof}
By naturality we only have to check the equality in the universal case. 
By using the maps (\ref{spspcomp1}) and (\ref{spspcomp2}) we can reduce the statement to the corresponding statement for spaces. 
The statement for spaces is proven in \cite[Proposition 4.1]{Du}. Roughly speaking, it follows by applying the formula   (\ref{randformel})  for the boundary orientation of a product to   $D^k \times D^l$. 
\end{proof}

We now study the multiplicative properties of the spectral sequence (\ref{ssbase}). We assume that we have towers of cofibrations of $R$-modules $X_{\LargerCdot}$, $Y_{\LargerCdot}$ and $W_{\LargerCdot}$.
 Furthermore, we assume that we have a pairing of towers $X_{\LargerCdot} \wedge_R Y_{\LargerCdot} \to W_{\LargerCdot}$ by which we mean 
 maps $X_n \wedge_R Y_k \to W_{n+k}$ such that the following diagrams commute in $\mathscr{M}_R$:
 \[ \begin{tikzcd}
     X_{n-1} \wedge_R Y_k \ar{d} \ar{r} & X_n \wedge_R Y_k     \ar{d} & X_n \wedge_R Y_{k-1} \ar{l} \ar{d} \\
     W_{n+k-1} \ar{r} & W_{n+k} & W_{n+k-1} \ar{l}.   
\end{tikzcd}\]
We get unique maps $X_n/ X_{n-1} \wedge_R Y_k/ Y_{k-1} \to W_{n+k}/W_{n+k-1}$ making the diagrams 
\begin{equation*} 
\begin{tikzcd}
X_n \wedge_R Y_k \ar{r} \ar{d} & W_{n+k}  \ar{d} \\
X_n/X_{n-1} \wedge_R Y_k/Y_{k-1} \ar{r} & W_{n+k}/W_{n+k-1} 
\end{tikzcd}
\end{equation*}
commutative. Let $E^*_{*,*}(X_{\LargerCdot})$, $E^*_{*,*}(Y_{\LargerCdot})$ and $E^*_{*,*}(W_{\LargerCdot})$  be the spectral sequences associated to 
the towers (see Lemma \ref{ss}). 
We get  maps 
\[\star: E^1_{n,m}(X_{\LargerCdot}) \otimes_{\mathbb{Z}} E^1_{k,l}(Y_{\LargerCdot}) \to E^1 _{n+k, m+l}(W_{\LargerCdot}).\]

In the following lemma we prove the compatibility of the product with the differentials for classes in  total degrees $\geq 3$. 
\begin{lem} \label{ss2}
Let $x \in E^1_{n,m}(X_{\LargerCdot}) \cap Z^r_{n,m}$ and $y \in E^1_{k,l}(Y_{\LargerCdot}) \cap Z^r_{k,l}$ be $r$-th cycles with $n+m \geq 3$ and $k+l \geq 3$.  Then, there are classes $\bar{x}  \in \pi_{n+m-1}(X_{n-r})$, 
$\bar{y} \in \pi_{k+l-1}(Y_{k-r})$ and $\bar{w} \in \pi_{n+m+k+l-1}(W_{n+k-r})$
such that the following holds:
\begin{itemize}
\item Under the maps $\pi_{n+m-1}(X_{n-r}) \to \pi_{n+m-1}(X_{n-1})$, $\pi_{k+l-1}(Y_{k-r}) \to \pi_{k+l-1}(Y_{k-1})$ and $
  \pi_{n+m+k+l-1}(W_{n+k-r}) \to \pi_{n+m+k+l-1}(W_{n+k-1}) $  
the classes $\bar{x}$, $\bar{y}$ and $\bar{w}$ get mapped to $\partial(x)$, $\partial(y)$ and $\partial(x \star y)$. 
\item Let $\tilde{x}$, $\tilde{y}$ and $\tilde{w}$ be the images of the classes $\bar{x}$, $\bar{y}$ and $\bar{w}$ under the maps
     \begin{align*}
     \pi_{n+m-1}(X_{n-r}) & \to \pi_{n+m-1}(X_{n-r}/X_{n-r-1}), \\ 
     \pi_{k+l-1}(Y_{k-r}) & \to \pi_{k+l-1}(Y_{k-r}/Y_{k-r-1}), \\
       \pi_{n+m+k+l-1}(W_{n+k-r}) & \to  \pi_{n+m+k+l-1}(W_{n+k-r}/W_{n+k-r-1}).
          \end{align*}
      Then the following equation holds:
      \[\tilde{w} = \tilde{x} \star y + (-1)^{n+m} x \star \tilde{y}.\] 
\end{itemize} 
\end{lem}
\begin{proof}
 Using Lemma \ref{compLES} and the fact that the maps $\mathbb{F}_R\Sigma^\infty S^{s-1} \to  \mathbb{F}_R\Sigma^\infty D^s$ satisfy the homotopy extension property, one gets  classes  $x' \in \pi_{n+m}(X_n, X_{n-r})$ and $y' \in \pi_{k+l}(Y_k , Y_{k-r})$ 
   that are mapped to $x$ and $y$ under 
   \begin{align*}
    \pi_{n+m}(X_n, X_{n-r}) \to    \pi_{n+m}(X_n, X_{n-1}) \xrightarrow{\phi}    \pi_{n+m}(X_n/X_{n-1})
   \end{align*}
   and 
   \begin{align*}
 \pi_{k+l}(Y_k, Y_{k-r}) \to    \pi_{k+l}(Y_k, Y_{k-1}) \xrightarrow{\phi}    \pi_{k+l}(Y_k/Y_{k-1}). 
 \end{align*}
We define $\bar{x} \coloneqq \kappa(x') \in \pi_{n+m-1}(X_{n-r})$ and $\bar{y} \coloneqq \kappa(y') \in \pi_{k+l-1}(Y_{k-r})$. 
By Lemma \ref{prodcomp2} the element $x \star y$ is the image of $ x' \cdot y'$ under 
\begin{align*}
  \pi_{n+m+k+l}(X_n \wedge_R Y_k, P') & \to  \pi_{n+m+k+l}(W_{n+k}, W_{n+k-r}) \\
& \to \pi_{n+m+k+l}(W_{n+k}, W_{n+k-1}) \to \pi_{n+m+k+l}(W_{n+k}/ W_{n+k-1})
 \end{align*}
 where $P':= (X_{n-r} \wedge_R Y_k) \, \amalg_{(X_{n-r} \wedge_R Y_{k-r})} \, (X_n \wedge_R Y_{k-r})$. 
We define $\bar{w}$ to be the image of $x' \cdot y'$ under 
\[ 
\pi_{n+m+k+l}(X_n \wedge_R Y_k, P') \to   \pi_{n+m+k+l}(W_{n+k}, W_{n+k-r}) \xrightarrow{\kappa} \pi_{n+m+k+l-1}(W_{n+k-r}).
\]
The claim then follows from Lemma \ref{compLES}, \ref{prodcomp2} and \ref{Leibniz}. 
\end{proof}

In order to show that Lemma \ref{ss2} holds for classes in arbitrary total degree, we next prove that  suspension isomorphisms give  isomorphisms of spectral sequences. 
Let $M$ be an $R$-module,  and let $X$  and $Y$ be pointed spaces. Analogously to $M \wedge X$,  we define $R$-modules $X \wedge M$ and $X \wedge M \wedge Y$ by applying the spectrification functor \cite[App.]{LMSt} to the prespectra given by
$(X \wedge M)(V) = X \wedge M(V)$ and
$(X \wedge M \wedge Y)(V) = X \wedge M(V) \wedge Y$. 
Recall  that we fixed isomorphisms 
\begin{equation*}
S^{m+1}_R \cong S^m_R \wedge S^1.
\end{equation*}
For $m \in \mathbb{Z}$ we define  $ S^{m+1}_R \cong S^1 \wedge S^m_R$ as the composition
\[
 S^{m+1}_R \cong S^m_R \wedge S^1 \cong S^1 \wedge S^m_R,
 \]
and we define  $S^{m+2}_R \cong S^1 \wedge S^m_R \wedge S^1$ as the composition
\[
 S^{m+2}_R \cong S^{m+1}_R \wedge S^1  \cong (S^1 \wedge S^m_R) \wedge S^1 \cong S^1 \wedge  S^m_R \wedge S^1.
 \]
These isomorphisms define group isomorphisms
\begin{align*}
\sigma_r: \pi_m(M) & \to \pi_{m+1}(M \wedge S^1), \\
\sigma_l: \pi_m(M) & \to \pi_{m+1}(S^1 \wedge M), \\
\sigma_b: \pi_m(M) & \to \pi_{m+2} (S^1 \wedge M \wedge S^1)
\end{align*}
such that 
\begin{equation} \label{sigmalr}
\begin{tikzcd}
\pi_m(M) \ar{r}{\sigma_r} \ar{rd}[swap]{\sigma_l} &  \pi_{m+1}(M \wedge S^1) \ar{d} \\
 &      \pi_{m+1}(S^1 \wedge M)
\end{tikzcd}
\end{equation}
and 
\begin{equation} \label{sigmablr}
\begin{tikzcd}
\pi_m(M) \ar{r}{\sigma_b} \ar{d}[swap]{\sigma_l} & \pi_{m+2}(S^1 \wedge M \wedge S^1) \ar{d}{\cong} \\
\pi_{m+1}(S^1 \wedge M) \ar{r}{\sigma_r} & \pi_{m+2}\bigl((S^1  \wedge M) \wedge S^1 \bigr) 
\end{tikzcd}
\end{equation} 
are commutative. 
By commutativity of (\ref{sigmalr}) and since the map 
\[\begin{tikzcd}
\pi_*\bigl(M \wedge (S^1 \wedge S^1)\bigr) \ar{r} & \pi_*(M \wedge (S^1 \wedge S^1)\bigr), 
\end{tikzcd}\]
interchanging the two  $S^1$'s,   is the negative of the identity map, 
we get that $\sigma_r \circ \sigma_l = - \sigma_l \circ \sigma r$. 
Using this and the commutativity of (\ref{sigmablr}) one gets
$\sigma_b \circ \sigma_r = - \sigma_r \circ \sigma_b$. 
\begin{lem}
Let $\cdots \to  X_{n-1} \to X_n \to X_{n+1} \to \cdots$ be a  sequence  of cofibrations of $R$-modules. 
The maps $\sigma_r$, $\sigma_l$, $\sigma_b$ 
define isomorphisms of unrolled exact couples and thus isomorphisms of spectral sequences 
\begin{align*}
 E_{*,*}^*(X_{\LargerCdot}) & \to E^*_{*,*+1}(X_{\LargerCdot} \wedge S^1), \\  
  E_{*,*}^*(X_{\LargerCdot}) & \to E^*_{*,*+1}(S^1 \wedge X_{\LargerCdot}), \\  
      E_{*,*}^*(X_{\LargerCdot}) & \to E^*_{*,*+2}(S^1 \wedge X_{\LargerCdot} \wedge S^1).  
     \end{align*}
\end{lem}
\begin{proof}
Consider for example the third  case. We only need to check that $\sigma_b$ and $\partial$ commute. Recall from the proof of Lemma \ref{ss} that there was a sign in the definition of $\partial$. 
The claim follows from the fact that for a cofibration of $R$-modules $f \colon M \to N$ and for  pointed spaces $X$  and $Y$ the diagram


\[\begin{tikzcd}
(X \wedge N \wedge Y)/(X \wedge M \wedge Y) \ar{d}{\cong} & C(\id_X \wedge f \wedge \id_Y) \ar{d}{\cong} \ar{l}[swap]{\simeq} \ar{r} & (X \wedge M \wedge Y) \wedge S^1  \ar{d}{\cong}\\
X \wedge (N/M) \wedge Y & X \wedge C(f) \wedge Y \ar{l}[swap]{\simeq} \ar{r} & X \wedge (M \wedge S^1) \wedge Y,
\end{tikzcd}
\]
commutes, 
from $\sigma_b \circ \sigma_r = - \sigma_r \circ \sigma_b$ and 
from the fact that interchanging two $S^1$'s induces the negative of the identity in homotopy. 
\end{proof}
We want to define pairings of suspensions of towers.
Note that for $R$-modules  $M$ and $N$ and for  pointed spaces $X$ and $Y$ we can define  a natural isomorphism
 \[X \wedge (M \wedge_R N) \wedge Y \cong (X \wedge M) \wedge_R (N \wedge Y). \]
Let $X_{\LargerCdot}$, $Y_{\LargerCdot}$ and $W_{\LargerCdot}$ be towers of cofibrations of $R$-modules. Suppose that we have a pairing of towers $X_{\LargerCdot} \wedge_R Y_{\LargerCdot} \to W_{\LargerCdot}$.  Then 
 the maps 
\[\begin{tikzcd}
 (X \wedge X_n) \wedge_R (Y_k \wedge Y)\ar{r}{\cong} &   X \wedge (X_n \wedge_R Y_k) \wedge Y \ar{r} & X \wedge W_{n+k} \wedge Y 
 \end{tikzcd}\]
define a pairing of towers $(X \wedge X_{\LargerCdot}) \wedge_R (Y_{\LargerCdot} \wedge Y) \to X \wedge W_{\LargerCdot} \wedge Y$. 

\begin{lem} \label{prodsusp}
The  diagram
\[\begin{tikzcd}
E^1_{n,m}(X_{\LargerCdot}) \otimes E^1_{k,l}(Y_{\LargerCdot}) \ar{r}{\star} \ar{d}[swap]{\sigma_l \otimes \sigma_r} & E^1_{n+k, m+l}(W_{\LargerCdot}) \ar{d}{\sigma_b} \\
E^1_{n, m+1}(S^1 \wedge X_{\LargerCdot}) \otimes E^1_{k, l+1}(Y_{\LargerCdot} \wedge S^1) \ar{r}{\star} & E^1_{n+k, m+l+2}(S^1 \wedge W_{\LargerCdot} \wedge S^1) 
\end{tikzcd} \]
commutes up to the sign $(-1)^{k+l}$. 
\end{lem}
\begin{proof}
Because of  (\ref{sigmablr}) and (\ref{sigmalr})  we only have to check  that the diagram
\[\begin{tikzcd}[column sep = small]
\pi_{s}(X_n/{X_{n-1}}) \otimes \pi_{t}(Y_k/{Y_{k-1}})  \MySymbo{ddr} \ar{r} \ar{dd}{\sigma_r \otimes \id} & \pi_{s+t}(X_n/{X_{n-1}}  \wedge_R Y_k/{Y_{k-1}})  \ar{d}{\sigma_r} \\
 & \pi_{s+t+1}\bigl((X_n/{X_{n-1}}  \wedge_R Y_k/{Y_{k-1}}) \wedge S^1\bigr) \ar{d}{\cong}\\
\pi_{s+1}(X_n/{X_{n-1}} \wedge S^1) \otimes \pi_{t}(Y_k/{Y_{k-1}})  \ar{d}{\cong} \ar{r} & \pi_{s+t+1}\bigl( (X_n/{X_{n-1}} \wedge S^1) \wedge_R Y_k/{Y_{k-1}}\bigr) \ar{d} {\cong}  \\
\pi_{s+1}(S^1 \wedge X_n/{X_{n-1}}) \otimes \pi_t(Y_k/{Y_{k-1}} ) \ar{dd}{\id \otimes \sigma_r} \ar{r} &  \pi_{s+t+1}\bigl((S^1 \wedge X_n/{X_{n-1}} ) \wedge_R Y_k/{Y_{k-1}}) \ar{d}{\sigma_r} \\
 &  \pi_{s+t+2}\Bigl(\bigl((S^1 \wedge X_n/{X_{n-1}}) \wedge_R Y_k/{Y_{k-1}}\bigr) \wedge S^1\Bigr) \ar{d}{\cong} \\
 \pi_{s+1}(S^1 \wedge X_n/{X_{n-1}}) \otimes \pi_{t+1}(Y_k/{Y_{k-1}} \wedge S^1)  \ar{r} & \pi_{s+t+2}\bigl( (S^1 \wedge X_n/{X_{n-1}}) \wedge_R (Y_k/{Y_{k-1}} \wedge S^1)\bigr) 
\end{tikzcd}\]
commutes up to the indicated sign. 
This follows from (\ref{susp1}) and (\ref{susp2}). 
\end{proof}

Using the above, one easily gets: 
\begin{lem} \label{ss3}
Lemma \ref{ss2} is true without the assumptions $n+m \geq 3$ and $k+l \geq 3$. 
\end{lem}

We can conclude: 
\begin{cor} \label{ssmult}
In the situation of Lemma \ref{ss2} the following properties hold: The map
\begin{align*}
\star: E^r_{n,m}(X_{\LargerCdot}) \otimes_{\mathbb{Z}} E^r_{k,l}(Y_{\LargerCdot}) & \to E^r_{n+k,m+l}(W_{\LargerCdot})  \\ \nonumber
[\alpha] \otimes [\beta] & \mapsto [\alpha \star \beta] 
\end{align*}
is well-defined.  The Leibniz rule holds:
\[d^r(x \star y) = d^r (x) \star y + (-1)^{n+m} x \star d^r(y).\]
For $r \geq 2$ the following diagram commutes:
\begin{equation*} \label{indprod}
\begin{tikzcd}
E^r_{*,*}(X_{\LargerCdot}) \otimes_{\mathbb{Z}} E^r_{*,*}(Y_{\LargerCdot}) \ar{d}[swap]{\cong} \ar{rr}{\star} & & E^r_{*,*}(W_{\LargerCdot}) \ar{d}{\cong} \\
H_*(E^{r-1}_{*,*}) \otimes_{\mathbb{Z}} H_*(E^{r-1}_{*,*}) \ar{r} &  H_*(E^{r-1}_{*,*} \otimes_{\mathbb{Z}} E^{r-1}_{*,*}) \ar{r}{H_*(\star)} & H_*(E^{r-1}_{*,*})
\end{tikzcd}
\end{equation*}
\end{cor}


Let $X_{\LargerCdot} \wedge_R Y_{\LargerCdot} \to W_{\LargerCdot}$ be pairing of towers of cofibrations. 
We  study the convergence properties  of the associated pairing  of spectral sequences:
 We have a product
 \begin{align*}
 \star:  \pi_n(\colim_i X_i) \otimes_{\mathbb{Z}} \pi_m(\colim_j Y_j) & \to \pi_{n+m}(\colim_i X_i \wedge_R \colim_j Y_j) \\
  & \xrightarrow{\cong} \pi_{n+m}(\colim_{i,j} X_i  \wedge_R Y_j)  \to \pi_{n+m}(\colim_k W_k) .
 \end{align*}
 
We define $F^n_s(X_{\LargerCdot})$ by the image of $\pi_n(X_s) \to \pi_n(\colim_i X_i)$. 
We get a filtration  
\[\dots \subseteq F^n_{s-1}(X_{\LargerCdot}) \subseteq F^n_s(X_{\LargerCdot}) \subseteq \cdots \]
 of $\pi_n(\colim_i X_i)$. This is the filtration we used implicitly in the proof of the convergence statement of  Lemma \ref{ss}.    Analogously, we define filtrations of the groups $\pi_m(\colim_j Y_j)$ and $\pi_{n+m}(\colim_k W_k)$.  

The product $\star$ respects the filtrations, so we get a product
\[F_s^n(X_{\LargerCdot})/F_{s-1}^n(X_{\LargerCdot}) \otimes_{\mathbb{Z}} F^m_t(Y_{\LargerCdot})/F^m_{t-1}(Y_{\LargerCdot}) \to F_{s+t}^{n+m}(W_{\LargerCdot})/F_{s+t-1}^{n+m}(W_{\LargerCdot}).    \]
By Corollary \ref{ssmult} the map 
\[ \star: E^1_{*,*}(X_{\LargerCdot}) \otimes_{\mathbb{Z}} E^1_{*,*}(Y_{\LargerCdot}) \to E^1 _{*, *}(W_{\LargerCdot})\] induces a product
\[\star: E^{\infty}_{*,*}(X_{\LargerCdot}) \otimes_{\mathbb{Z}} E^{\infty}_{*,*}(Y_{\LargerCdot}) \to E^{\infty}_{*,*}(W_{\LargerCdot}) .\]
These two products are compatible: 

\begin{lem} \label{convmult}
Suppose  that we have $X_n = Y_n = W_n = \ast$ for $n < 0$, so that the spectral sequences converge. 
Then the following diagram commutes: 
\[ \begin{tikzcd}
F_s^n(X_{\LargerCdot})/F_{s-1}^n(X_{\LargerCdot}) \otimes F^m_t(Y_{\LargerCdot})/F^m_{t-1}(Y_{\LargerCdot}) \ar{r}{\star} \ar{d}[swap]{\cong} & F_{s+t}^{n+m}(W_{\LargerCdot})/F_{s+t-1}^{n+m}(W_{\LargerCdot}) \ar{d}{\cong} \\
E^{\infty}_{s, n-s}(X_{\LargerCdot}) \otimes E^{\infty}_{t, m-t}(Y_{\LargerCdot}) \ar{r}{\star} & E^{\infty}_{s+t, n+m-(s+t)}(W_{\LargerCdot}). 
\end{tikzcd}\]
\end{lem}
\begin{proof}
This easily follows from the definitions of the vertical isomorphisms (see \cite[Lemma 5.6]{Boar}).
\end{proof}

\subsection{The Atiyah-Hirzebruch spectral sequence} \label{AHsec}

In this subsection we use the results of the previous subsection to derive the existence of the Atiyah-Hirzebruch 
spectral sequence and some of its properties, especially its multiplicativity. 
It is stated in \cite[Section IV.3]{EKMM} that the Atiyah-Hirzebruch spectral sequence  has the expected multiplicative properties. We  work this out in more detail. 


\begin{rmk} \label{postcomalg}
Recall that the Eilenberg-Mac Lane spectrum $HA$ of a commutative ring $A$ can be realized as a commutative $S$-algebra: 
Multiplicative infinite loop space theory provides a functor from commutative rings to $E_{\infty}$ ring spectra  (\cite{Maywhat}, \cite{MayMult}, \cite{MayRingSp}) 
 and  the functor $S \wedge_{\mathscr{L}} - $ induces a functor from $E_{\infty}$ ring spectra to commutative $S$-algebras (\cite[Corollary II.3.6]{EKMM}). Since we can functorially replace commutative $S$-algebras by cofibrant ones, we can assume that 
 $HA$ is cofibrant. 
 
If  $M$ is an $A$-module, then by \cite[Section IV.2]{EKMM}  the Eilenberg-Mac Lane spectrum $HM$ can be realized as an $HA$-module. Eilenberg-Mac Lane spectra are unique:
If  $R$ is a connective ($(-1)$-connected) commutative $S$-algebra and  if $X$ and $Y$ are $R$-modules with $\pi_*(X) = 0 = \pi_*(Y)$ for  $ * \neq s$, then, by \cite[p.3]{LewMand}, we have an isomorphism 
\begin{equation}  \label{uniquems}
 \mathscr{D}_{R}(X, Y) \cong \Hom_{\pi_0(R)}\bigl(\pi_s(X), \pi_s(Y)\bigr). 
\end{equation}

\begin{rmk}
Let $R$ be a connective  commutative $S$-algebra. By \cite[Proposition IV.3.1]{EKMM}
there is a morphism  
\begin{equation*} \label{Postnikov}
R \to H\pi_0(R)
\end{equation*}
 in $\bar{h} \mathscr{C}\mathscr{A}_S$ that induces the identity on $\pi_0$. The
proof shows that it is natural: If $R \to R'$ is a morphism of commutative $S$-algebras, then 
\begin{equation} \label{postnikovnat}
 \begin{tikzcd}
R \ar{r} \ar{d} & H\pi_0(R) \ar{d}\\
R' \ar{r} & H\pi_0(R')
\end{tikzcd}
\end{equation}
is commutative in $\bar{h}\mathscr{C}\mathscr{A}_S$ .

Let $R$ be a connective cofibrant commutative $S$-algebra, so that we have a map of commutative $S$-algebras $R \to H\pi_0(R)$ realizing the identity on $\pi_0$. 
Let $G$ be a connective $R$-module. Then, we have a map $G \to H\pi_0(G)$ in $\mathscr{D}_R$  realizing the identity map on $\pi_0$ (compare \cite[Theorem II.4.13]{Ru}). Here, $H\pi_0(G)$ is an an $R$-module by pulling back the $H\pi_0(R)$-action along the map $R \to H\pi_0(R)$. 
\end{rmk}

\end{rmk}



We now prove the existence of the Atiyah-Hirzebruch spectral sequence. 
Recall from (\cite[Definition IV.1.7]{EKMM}) that for an $S$-algebra $R$, a right $R$-module $E$ and a left $R$-module $M$ one defines 
\[E^R_n(M) \coloneqq \pi_n(E \wedge_R^L M).\]
Sometimes we  write $E_m$ for $\pi_m(E)$. 
\begin{thm} \label{AHss}
Let $R$ be a connective  cofibrant commutative $S$-algebra. Let $M$ be a connective $R$-module and let $G$ be an arbitrary $R$-module. 
Then, we have a strongly convergent spectral sequence of the form
\[ E^2_{n,m} = (HG_m)^R_nM \Longrightarrow G^R_{n+m}M.\]
Here, $HG_m$ is an $R$-module by pulling back the $H\pi_0(R)$-action along $R \to H\pi_0(R)$. 
\end{thm}
\begin{proof}
This is Theorem IV.3.7 in \cite{EKMM}. Since we will need it later, we explain the proof in more detail. 
We replace $G$ by a cell approximation and denote this approximation again by $G$. 
Since $R$ is connective there is a CW $R$-module $\Gamma M$ and a weak equivalence of $R$-modules $\Gamma M \to M$ (\cite[Theorem III.2.10]{EKMM}). Because $M$ is connective we can assume that the $n$-skeleton $\Gamma M^n$ is  $\ast$ for $n < 0$ (\cite[proof of Theorem IV.3.6]{EKMM}). 
The $(n+1)$-skeleton $\Gamma M^{n+1}$ is built from the $n$-skeleton by attaching $(n+1)$-cells. Since the coproduct of cofibrations is a cofibration and since cofibrations are stable under cobase change,  $\Gamma M^n \to \Gamma M^{n+1}$ is a cofibration. Because $G \wedge_R (-)$ is left adjoint,
it preserves cofibrations. Hence, 
\[G \wedge_R \Gamma M^n \to G \wedge_R \Gamma M^{n+1}\] is a cofibration, too. 
Therefore, Lemma \ref{ss} gives a strongly convergent spectral sequence of the form
\[ E^1_{n,m} = \pi_{n+m}(G \wedge_R \Gamma M^n/{\Gamma M^{n-1}}) \Longrightarrow \pi_{n+m}(\colim_n G \wedge_R \Gamma M^n).\]
We have  $\colim_n G \wedge_R \Gamma M^n = G \wedge_R \colim_n \Gamma M^n = G \wedge_R \Gamma M$. Thus, it remains to identify the $E^2$-page.  The quotient  $\Gamma M^n/{\Gamma M^{n-1}}$ is a wedge of $n$-spheres $S^n_R$, so
by \cite[Proposition III.3.9]{EKMM} the K\"unneth map  induces an isomorphism
\[ G_m \otimes_{R_0} \pi_n(\Gamma M^n/\Gamma M^{n-1}) \to \pi_{n+m}(G \wedge_R \Gamma M^n/\Gamma M^{n-1}).\] 
Since the  K\"unneth map is compatible with the  suspension, the differential $d^1: E^1_{n,m} \to E^1_{n-1,m}$ identifies with $(-1)^{n+m}$ times the map
\begin{align} \label{d}
\begin{split}
 G_m \otimes_{R_0} \pi_n(\Gamma M^n/{\Gamma M^{n-1}}) & \xrightarrow{\cong}  G_m \otimes_{R_0} \pi_n(C)  \xrightarrow{\id \otimes d}  G_m \otimes_{R_0} \pi_{n-1}(\Gamma M^{n-1}) \\
  & \xrightarrow{}  G_m \otimes_{R_0} \pi_{n-1}(\Gamma M^{n-1}/\Gamma M^{n-2}), 
  \end{split}
 \end{align}
 where $C = C(\Gamma M^{n-1} \to \Gamma M^n)$.  We now fix $m$. 
 The above argument shows that there is a strongly convergent spectral sequence 
 \[ _mE^1_{n,k}  \Longrightarrow (HG_m)_{n+k}^R(M) \]
 whose $E^1$-page is given by
 \[ _mE^1_{n,k} \cong \begin{cases} 
                     0, & \text{~if~} k \neq 0; \\
                     G_m \otimes_{R_0}  \pi_n(\Gamma M^n/{\Gamma M^{n-1}}), & \text{~if~} k = 0;
                     \end{cases}\]
    and whose differential $d^1: \ _mE^1_{n,0} \to \ _mE^1_{n-1, 0}$  is $(-1)^n$ times the map (\ref{d}). 
    Since this  spectral sequence has to collapse at the $E^2$-page  we get
    \[E^2_{n,m} \cong \ _mE^2_{n,0} \cong (HG_m)_n^R(M).\qedhere\]
    \end{proof}

\begin{lem} \label{natati}
Let $R$ be a connective cofibrant commutative $S$-algebra. Let $M \to \bar{M}$  and $G \to \bar{G}$ be morphisms in $\mathscr{D}_R$, where $M$ and $\bar{M}$ are connective.  Then, we have a morphism of Atiyah-Hirzebruch spectral sequences between the spectral sequence converging to  $G^R_{*}M$ and the spectral sequence converging to $\bar{G}^R_*\bar{M}$. The map $G^R_{*}M \to \bar{G}^R_*\bar{M}$ respects the filtrations and the induced map on the associated graded identifies with the induced map on $E^{\infty}$-pages. On the $E^2$-pages the map is the obvious one. 
\end{lem}
\begin{proof}
This follows easily from the fact that maps between CW modules over a connective $S$-algebra are homotopic to cellular maps \cite[Theorem III.2.9]{EKMM}. 
\end{proof}

\begin{lem} \label{edge2}
In the situation of  Theorem \ref{AHss} we have the following: If $G$ is connective, the edge homomorphism 
\[\pi_n(G \wedge_R^L M) = F_n^n \to F^n_n/{F^n_{n-1}} \xrightarrow{\cong} E^{\infty}_{n,0} \hookrightarrow E^2_{n,0} \cong \pi_n(HG_0 \wedge_R^L M)\]
is the map induced by $G \to HG_0$.
\end{lem}
\begin{proof}
This follows by applying Lemma \ref{natati}  to the morphisms $M \xrightarrow{\id} M$ and  $G \to HG_0$ and by observing that the edge homomorphism of the Atiyah-Hirzebruch spectral sequence converging to $\pi_*(HG_0 \wedge_R^L M)$ is the identity. 
\end{proof}
 \begin{lem} \label{natatichanger}
 Let $R \to R'$ be a  morphism of connective cofibrant commutative $S$-algebras.  Let $M$ be a connective $R'$-module and let $G$ be an $R'$-module. Then, we  a morphism of Atiyah-Hirzebruch spectral sequences converging to the  map $G^R_*M \to G^{R'}_*M$. On the $E^2$-pages the map is given by the canonical maps $(HG_m)^R_*M \to (HG_m)^{R'}_*M$. Here, note that we can realize $HG_m$ as $H\pi_0(R')$-module and consider it as an $R$-module via the map $R \to R' \to H\pi_0(R')$, and  we can realize $HG_m$ as $H\pi_0(R)$-module and consider it as an $R$-module via $R \to H\pi_0(R)$. These two $R$-modules  are isomorphic by  (\ref{uniquems}) .
 \end{lem}
 \begin{proof}
 Let $\Gamma^RG$ and $\Gamma^{R'}G$ be the cell  approximations of $G$ used in the construction  of the spectral sequences and  let $\Gamma^R M$ and  $\Gamma^{R'}M$ be the CW-approximation of $M$ used in the construction of the spectral sequences.  Let $\Gamma^RG \to  \Gamma^{R'}G$ and  $\Gamma^R M \to \Gamma^{R'}M$ be  maps in $\mathscr{M}_R$ lifting the identity maps up to homotopy. 
  Because of $(\Gamma^RM)^n = * = (\Gamma^{R'}M)^n$ for $n < 0$,  one can show, as in the proof of the existence of cellular approximations of maps between CW-complexes (see \cite[Theorem 8.5.4]{Dieckbook}, \cite[Lemma II. 3.2]{whi}), that $\Gamma^RM \to \Gamma^{R'}M$ is homotopic in $\mathscr{M}_R$ to a map that respects the skeleton filtrations of $\Gamma^RM$ and $\Gamma^{R'}M$.  
 For this note that the only properties of the sequence  of cofibrations of $R$-modules
 \[  \cdots \to (\Gamma^{R'}M)^{n-1} \to (\Gamma^{R'}M)^n \to (\Gamma^{R'}M)^{n+1} \to \cdots \to \Gamma^{R'}M \] 
 that are needed in the proof are that the $n$-th  homotopy group of the $R$-module homomorphism $(\Gamma^{R'} M)^n\to \Gamma^{R'}M$ is zero for $n \geq 1$ and that $\pi_0((\Gamma^ {R'}M)^0) \to \pi_0(\Gamma^{R'} M)$ is surjective. 
 Both follows from the fact that $\pi_n(\Gamma^{R'}M/(\Gamma^{R'}M)^n) = 0$ by \cite[p.57]{EKMM}.

 It is clear that the maps
 \[\begin{tikzcd}
 \Gamma^RG \wedge_R (\Gamma^R M)^n \ar{r} & \Gamma^{R'}G \wedge_R (\Gamma^{R'} M)^n \ar{r} & \Gamma^{R'}G \wedge_{R'} (\Gamma^{R'} M)^n
 \end{tikzcd}\]
  induce a morphism of unrolled exact couples and therefore a map of spectral sequences which converges to the  map $G^R_*M \to G^{R'}_*M$. 
  Note that  map
  \[\begin{tikzcd}
  \pi_{n+m}\bigl((\Gamma^R G \wedge_R (\Gamma^RM)^n/{(\Gamma^RM)^{n-1}}\bigr) \ar{r} &  \pi_{n+m}\bigl((\Gamma^{R'} G \wedge_{R'} (\Gamma^{R'}M)^n/{(\Gamma^{R'}M)^{n-1}}\bigr)
  \end{tikzcd} \]
  identifies with the map
  \[\begin{tikzcd}
  G_m \otimes_{R_0} \pi_n\bigl((\Gamma^RM)^n/{(\Gamma^RM)^{n-1})}  \ar{r} & G_m \otimes_{R'_0} \pi_n\bigl((\Gamma^{R'}M)^n/{(\Gamma^{R'}M)^{n-1}\bigr)}.
  \end{tikzcd}\]
Applying the proven statements as well  to the $R'$-modules $HG_m$, one gets the identification of the map of spectral sequences on $E^2$-pages.  For this also note that the isomorphism between the two realization of $HG_m$ induces a morphism of spectral sequences by Lemma \ref{natati}.
 \end{proof}
We now prove the  multiplicativity of the Atiyah-Hirzebruch spectral sequence: 
\begin{prop} \label{AHssM}
Let $R$ be a connective cofibrant commutative $S$-algebra. 
Let $M$, $N$ and $L$ be  connective $R$-modules and  let $G$ be an $R$-module. Let $_M E^*_{*,*}$, $_NE^*_{*,*}$ and $_LE^*_{*,*}$ be the  spectral sequences defined in Theorem \ref{AHss}.  Then  maps $G \wedge_R^L G \to G$ and $M \wedge_R^L N \to L$ in $\mathscr{D}_R$ induce  a pairing of spectral sequences 
\[ _M E^*_{*,*} \otimes_{\mathbb{Z}}  {_NE^*_{*,*}} \to   {_LE^*_{*,*}}  \]
that converges to the product
\[ \pi_*(G \wedge_R^L M) \otimes_{\mathbb{Z}} \pi_*(G \wedge_R^L N) \to \pi_*(G \wedge_R^L L).\]
\end{prop}
\begin{proof}
Since the spectral sequences were constructed using a cell approximation of $G$, we can assume that the product $G \wedge_R^L G \to G$  in $\mathscr{D}_R$  can be represented by a map $G \wedge_R G \to G$ in $\mathscr{M}_R$.  Let $\Gamma M$, $\Gamma N$ and $\Gamma L$ be the CW-approximations of $M$, $N$ and $L$ used in the definition of the spectral sequences. Then,  the map $M \wedge_R^L N \to L$ can be represented by a  map
\[ \Gamma M \wedge_R \Gamma N \to \Gamma L\]
in $\mathscr{M}_R$ that is cellular. 
We define maps  $(G \wedge_R \Gamma M^n) \wedge_R ( G \wedge_R \Gamma N^k) \to G \wedge \Gamma L^{n+k}$
by 
\begin{align*}
(G \wedge_R \Gamma M^n) \wedge_R ( G \wedge_R \Gamma N^k) & \xrightarrow{\cong}  (G \wedge_R G) \wedge_R (\Gamma M^n \wedge_R \Gamma N^k) \\ &  \to G \wedge (\Gamma M \wedge_R \Gamma N)^{n+k} 
 \to G \wedge_R \Gamma L^{n+k}.
\end{align*}
Since the  diagrams 
\[\begin{tikzcd}
\Gamma M^{n-1} \wedge_R \Gamma N^k \ar{r} \ar{d} & \Gamma M^n \wedge_R \Gamma N^k \ar{d}  & \Gamma M^n \wedge_R \Gamma N^{k-1} \ar{l} \ar{d} \\
(\Gamma M \wedge_R  \Gamma N)^{n+k-1} \ar{r} & (\Gamma  M \wedge_R \Gamma N)^{n+k} & (\Gamma M \wedge_R \Gamma N)^{n+k-1} \ar{l}  
\end{tikzcd}\]
and 
\[ \begin{tikzcd}
(\Gamma M \wedge_R \Gamma N)^{n+k-1} \ar{r} \ar{d} & (\Gamma M \wedge_R \Gamma N)^{n+k} \ar{d} \\
\Gamma L^{n+k-1} \ar{r} & \Gamma L^{n+k}  
\end{tikzcd}\]
commute in $\mathscr{M}_R$, the proposition follows from Corollary  \ref{ssmult} and  Lemma \ref{convmult}.
\end{proof}
\section{The Brun spectral sequence} \label{Brunsubsec}
The goal of this section is to apply the results of the previous section to  construct a generalization of  the following spectral sequence of  Brun (see \cite[Theorem 6.2.10]{Brun}):

\begin{thm}[Brun]
Let $R \to R_1$ be a ring homomorphism of commutative rings. 
There is a multiplicative spectral sequence of the form
\[E^2_{n,m} = \pi_n\THH\bigl(HR_1, H\Tor^R_m(R_1,R_1)\bigr) \Longrightarrow \pi_{n+m}\THH(HR, HR_1).\]
\end{thm}

We first recall two definitions of topological Hochschild homology.  For this, recall the definition of geometric realization.  

  \begin{rmk} \label{georeal}
  Let $X_{\LargerCdot}$ be a simplicial object in $\mathscr{M}_S$ or $\mathscr{S}$. Then, its geometric realization 
  $|X_{\LargerCdot}|$ is defined as the coend 
\[\int^{\Delta} X_q \wedge (\Delta_q)_+,\]
where $\Delta_q$ denotes the standard topological $q$-simplex. 
 Since tensors and colimits in $\mathscr{M}_S$ are created in $\mathscr{S}$ geometric realization in $\mathscr{M}_S$ is given 
  by geometric realization in  $\mathscr{S}$.
  Now, let $X_{\LargerCdot}$ be a simplicial commutative $S$-algebra. Then, by \cite[Proposition X.1.5]{EKMM} its geometric realization can be endowed with the structure of a commutative $S$-algebra. 
  We can also form the internal geometric realization 
  \[ \int^{\Delta} X_q \otimes \Delta_q   \, \, \in \mathscr{C}\mathscr{A}_S.\]
  By  \cite[Proposition VII.3.3]{EKMM} these two commutative $S$-algebras are isomorphic.

  
\end{rmk}
 \begin{defi}[$\THH$, first definition] \label{THH1}
 Let $A$ be a cofibrant commutative $S$-algebra. Let $B$ be an $(A,A)$-bimodule. 
The  simplicial $S$-module  $\THH(A;B)_{\LargerCdot}$ is defined  by
 \[ \THH(A; B)_n \coloneqq B \wedge_S \underbrace{ A \wedge_S \cdots \wedge_S A}_{n \text{~times~}}\]
 together with the usual face and degeneracy maps (see \cite[Definition IX.2.1]{EKMM}). 
 We define  \[\THH(A;B) \coloneqq |\THH(A;B)_{\LargerCdot}| \]  
We  write $\THH(A)$ for $\THH(A;A)$.  We denote the homotopy groups of $\THH(A;B)$ and $\THH(A)$ by $\THH_*(A;B)$ and $\THH_*(A)$. 
  \end{defi}


  \begin{defi}[$\THH$, second definition] \label{THH2}
  Let $A$ be a cofibrant commutative $S$-algebra and let $B$ be an $(A, A)$-bimodule. 
  Let $A^e = A \wedge_S A^{op}$ be the enveloping algebra of $A$.  We define 
  \[ \THH(A;B) = B \wedge_{A^e}^L A \in \mathscr{D}_{A^e}.\] 
  \end{defi}

  By \cite[Proposition IX.2.5]{EKMM} the image of $B \wedge_{A^e}^L A$ under $\mathscr{D}_{A^e} \to \mathscr{D}_S$ is isomorphic 
  to the image of $|n  \mapsto B \wedge_S  A^{\wedge_S n}|$ under $\mathscr{M}_S \to \mathscr{D}_S$, if $B$ is a cell $A^e$-module.

   We now study  algebra structures for both definitions of $\THH$ in the case  where $B$ is a commutative $A$-algebra. In Lemma \ref{compthh} we will see  that the isomorphism of 
    \cite[Proposition IX.2.5]{EKMM}  preserve multiplicative structures. 
 
 \begin{rmk}
  Let $A$ be a cofibrant commutative $S$-algebra.   Let  $B$ be a commutative $A$-algebra. Then, $[n] \mapsto B \wedge_S A^{\wedge_S n}$ is a simplicial commutative $S$-algebra. Thus, by Remark \ref{georeal}, its geometric realization is  an object in $\mathscr{C}\mathscr{A}_S$.  Moreover, the maps $B \to B \wedge_S A^{\wedge_S n}$ and $B \wedge_S A^{\wedge_S n} \to B$  endow $|[n] \mapsto B \wedge_S A^{\wedge_S n}|$  with the structure of an augmented commutative $B$-algebra.  Here, note that by the universal property of coends we have $|\underline{B}| \cong B$, where $\underline{B}$ denotes the constant simplicial object. 

Note that the $(A,A)$-bimodule  structures on $A$ and $B$ come from  the  $A^e$-algebra structures given by the maps of commutative $S$-algebras $A^e = A \wedge_S A \to A$ and 
\[( A^e \to A \to B) = (A^e \to B^e \to B).\] 
Since the smash product of two ring spectra is a ring spectrum we get  that $B \wedge_{A^e}^L A$ is an $A^e$-ring spectrum.
  \end{rmk}

\begin{lem} \label{compthh}
Let $A \to B$ be a morphism of cofibrant commutative $S$-algebras. 
Then the image of $B \wedge_{A^e}^L A$ under $\mathscr{D}_{A^e} \to \mathscr{D}_S$ 
and the image of $| [n] \mapsto B \wedge_S A^{\wedge_S n}|$ under $\mathscr{M}_S \to \mathscr{D}_S$ are isomorphic as $S$-ring spectra.  
\end{lem}
\begin{proof}

 

Let $B^S(A,A,A)_{\LargerCdot}$ be the bar construction (see \cite[Definition IV.7.2]{EKMM}).  It is a simplicial commutative $S$-algebra with $B^S(A, A, A)_n = A \wedge_S A^{\wedge_S n} \wedge_S A$. The inclusions of the first and the last smash factor   
give a map of commutative $S$-algebras
\[ A^e \cong  |\underline{A^e}|  \to |B^S(A,A,A)_{\LargerCdot}|.\]  We have: 
\begin{itemize}
\item By \cite[Lemma VII.7.3]{EKMM} the map of commutative $A^e$-algebras 
\[ |B^S(A,A,A)_{\LargerCdot}| \to |\underline{A}| \cong A\]
 is a weak equivalence. 
\item The map $A^e \to |B^S_{\LargerCdot}(A,A,A)|$ is a cofibration of commutative $S$-algebras. 
To see this we consider the simplicial $1$-simplex $[n] \mapsto \Delta^1_n$. We interpret $\Delta^1$ as a discrete simplicial space. Since the functor $A \otimes -: \mathscr{U} \to \mathscr{C}\mathscr{A}_S$ commutes with coproducts,  we can identify $|B_{\LargerCdot}^S(A,A,A)|$ with the geometric realization of  the simplicial commutative $S$-algebra
\[[n] \mapsto A \otimes (\Delta^1)_n.\]
  By \cite[Proposition VII.3.2]{EKMM} this is isomorphic  in $\mathscr{C}\mathscr{A}_S$ to $A \otimes |\Delta^1|$. Likewise, $|\underline{A^e}|$ is isomorphic to $A \otimes |\partial \Delta^1|$. The map $|\underline{A^e}| \to |B^S(A,A,A)_{\LargerCdot}|$ can be identified with
\[ \begin{tikzcd}
A \otimes |\partial \Delta^1 | = \bigl(A \otimes |\partial \Delta^1 |  \wedge_{S \otimes |\partial \Delta^1 | } S \otimes |\Delta^1|\bigr)  \ar{r} &  A \otimes |\Delta^1| ,
\end{tikzcd}\]
which is a cofibration in $\mathscr{C}\mathscr{A}_S$, because the model category structure  is topological. 
\end{itemize}
We are now ready to prove the lemma. We have  an isomorphism of commutative $S$-algebras 
\begin{eqnarray*}
|[n] \mapsto B \wedge_S A^{\wedge_S n}| &  \cong &  |\underline{B} \wedge_{\underline{A^e}} B^S(A,A,A)_{\LargerCdot}| 
\end{eqnarray*}
By  the universal properties of coends and pushouts, and the fact that the  functor $- \otimes X: \mathscr{C}\mathscr{A}_S \to \mathscr{C}\mathscr{A}_S$ commutes with colimits, we get that this is isomorphic to 
 \[ |\underline{B}| \wedge_{|\underline{A^e}|} |B^S(A,A,A)_{\LargerCdot}|   \cong  B \wedge_{A^e} |B^S(A,A,A)_{\LargerCdot}|.\] 
Since $A^e$ is a cofibrant commutative $S$-algebra (as coproduct of cofibrant objects), we get by Lemma \ref{cof1en}
that the map
\[\begin{tikzcd}
 B \wedge_{A^e}^L |B^S(A,A,A)_{\LargerCdot}|  \ar{r} & B \wedge_{A^e} |B^S(A,A,A)_{\LargerCdot} |
 \end{tikzcd}\]
 in $\mathscr{D}_{A^e}$  is an isomorphism of $A^e$-ring spectra.  
Since
\[   B \wedge_{A^e}^L |B^S(A,A,A)_{\LargerCdot} | \cong B \wedge_{A^e}^L A\]
as $A^e$-ring spectra, this finishes the proof. 

\end{proof}
\begin{rmk} \label{natTHH}
Let 
\begin{equation} \label{standdiag}
\begin{tikzcd}
A \ar{r} \ar{d} & B \ar{d}  \\
C \ar{r} & D 
\end{tikzcd}
\end{equation}
be a commutative diagram of cofibrant commutative $S$-algebras. We then have a commutative diagram in $\mathscr{D}_S$ 
\[\begin{tikzcd}
B \wedge_{A^e}^L A \ar{r}{\cong} \ar{d} &  \vert \left [ n \right ]  \mapsto B \wedge_S A^{\wedge_S n}  \vert \ar{d} \\
D \wedge_{C^e}^L C \ar{r}{\cong} &   \vert \left [ n \right ]  \mapsto D \wedge_S C^{\wedge_S n} \vert,
\end{tikzcd}\]
where the left vertical map is given by $B \wedge_{A^e}^L A \to  D \wedge_{A^e}^L C \to D \wedge_{C^e}^L C$. 
\end{rmk}
Until the end of the section, we work with the  Definition (\ref{THH2}) for $\THH$ (second definition). 
In order to proof the generalization of Brun's spectral sequence,  we first identify $\THH$ with another ring spectrum. After that, we apply the Atiyah-Hirzebruch spectral sequence. 
\begin{lem} \label{THHid}
Let $S \to A \to B$ be cofibrations of commutative $S$-algebras.
Then we have an isomorphism of $A^e$-ring spectra 
\[\THH(A;B) \cong (B \wedge_A B) \wedge_{B^e}^L B.\] 
\end{lem}
\begin{proof}
We factor the map $B^e \to B$ in commutative $S$-algebras as
\[\begin{tikzcd}
B^e    \ar[tail]{r} & \tilde{B} \ar[two heads]{r}{\sim} & B.
\end{tikzcd}\]
We then have an isomorphism of $A^e$-ring spectra 
\[ \THH(A;B) \cong \tilde{B} \wedge_{A^e}^L A.\] 
We claim that the map 
\[\begin{tikzcd}
\tilde{B} \wedge_{A^e}^L A  \ar{r} & \tilde{B} \wedge_{A^e} A
\end{tikzcd}\] 
in $\mathscr{D}_{A^e}$ is an isomorphism of $A^e$-ring spectra.  By Lemma \ref{cof1en} we only have to show that 
$A^e \to B^e$ is a cofibration of commutative $S$-algebras. 
This is true because $A \to  B$ is a cofibration, and, by denoting the coproduct inclusions by $i_1$ and $i_2$, we have pushout diagrams
\[\begin{tikzcd}
A \ar[tail]{d} \ar{r}{i_2} & A \wedge_S A \ar{d} &  & A \ar[tail]{d} \ar{r}{i_1} & A \wedge_S B \ar{d}  \\
B \ar{r}{i_2} & A \wedge_S B  & &  B \ar{r}{i_1} & B \wedge_S B.
\end{tikzcd}\]
Since $\tilde{B} \cong B^e \wedge_{B^e} \tilde{B}$ as $B^e$-algebras, we  have an isomorphism of $A^e$-algebras
\[ \tilde{B} \wedge_{A^e} A \cong A \wedge_{A^e} \tilde{B} \cong A \wedge_{A^e} (B^e \wedge_{B^e} \tilde{B}).\]
By using the universal property of pushouts in the category of commutative $S$-algebras, one sees that
\[  A \wedge_{A^e} (B^e \wedge_{B^e} \tilde{B}) \cong (A \wedge_{A^e} B^e) \wedge_{B^e} \tilde{B}.\]
Obviously, the isomorphism is  an isomorphism of $A^e$-algebras. 
Again, by using the universal property of pushouts  one  sees that we have an isomorphism of commutative $S$-algebras 
\[ A \wedge_{A^e} B^e \cong B \wedge_A B.\]
The isomorphism  is an isomorphism of $B^e$-algebras. Therefore, we get 
\[(A \wedge_{A^e} B^e) \wedge_{B^e} \tilde{B} \cong (B \wedge_A B) \wedge_{B^e} \tilde{B} \]
as $B^e$-algebras.
We claim that the map 
\[\begin{tikzcd}
(B \wedge_A B) \wedge_{B^e}^L \tilde{B}  \ar{r} & (B \wedge_A B) \wedge_{B^e} \tilde{B} 
\end{tikzcd}\]
in $\mathscr{D}_{B^e}$ is an isomorphism of $B^e$-ring spectra. This follows from Lemma \ref{cof1en} if we can show that $B \wedge_A B$ is a cofibrant commutative $S$-algebra. 
But this is clear, because $A \to B \wedge_A B$ is a cofibration in $\mathscr{C}\mathscr{A}_S$, since cofibration are stable under cobase change. 
As we have \[(B \wedge_A B) \wedge_{B^e}^L \tilde{B} \cong (B \wedge_A B) \wedge_{B^e}^L B\] 
as $B^e$-ring spectra, this finishes the proof. 
\end{proof}
\begin{rmk} \label{natTHHid}
Given a diagram of the form (\ref{standdiag}), where the horizontal maps are cofibrations in $\mathscr{C}\mathscr{A}_S$,   one gets a commutative diagram in $\mathscr{D}_S$ 
\[\begin{tikzcd}
\THH(A;B) \ar{r}{\cong} \ar{d} \ar{d}  & (B \wedge_A B) \wedge_{B^e}^L B \ar{d} \\
\THH(C;D) \ar{r}{\cong} & (D \wedge_C D) \wedge_{D^e}^L D.
\end{tikzcd}\]
Here, the right vertical map is given by 
\[\begin{tikzcd}
(B \wedge_A B) \wedge_{B^e}^L  B \ar{r} & (D \wedge_C D) \wedge_{B^e}^L D \to (D \wedge_C D) \wedge_{D^e}^L D.
\end{tikzcd}\]
One uses that maps in $\mathscr{C}\mathscr{A}_S$  can be factored functorially into cofibrations followed by weak equivalences. 
\end{rmk}
Before we can prove the generalization of Brun's spectral sequence we need another remark:
 
 \begin{rmk} \label{assocSR}
 Let $A$ be a  commutative $S$-algebra and let $E$ be an  $S$-ring spectrum that is a cell $S$-module. Then by  \cite[Proposition III.4.1]{EKMM}  $E \wedge_S A$ is a cell $A$-module. We can equip $E \wedge_S A$ with the structure of an  $A$-ring spectrum.  Its image under $\mathscr{D}_A \to \mathscr{D}_S$ is isomorphic to the $S$-ring spectrum $E \wedge_S^L A$.   If $A$ is a commutative $R$-algebra and $B$ is an $R$-ring spectrum,  we have
\begin{equation*} \label{ringsp}
 (E \wedge_S A) \wedge_R^L B \cong E \wedge_S^L (A \wedge_R^L B)
\end{equation*}
as $S$-ring spectra. The isomorphism is natural in $E$, $A$ and $B$. If $R \to R'$ is a morphism of commutative $S$-algebras, $B$ is an $R'$-ring spectrum and $A$ is a commutative $R'$-algebra the following diagram commutes in $\mathscr{D}_S$:
\[\begin{tikzcd}
(E \wedge_S A) \wedge_R^L B \ar{r}{\cong}  \ar{d} & E \wedge_S^L (A \wedge_R^LB)  \ar{d} \\ 
(E \wedge_SA)  \wedge_{R'}^L B \ar{r}{\cong}  & E \wedge_S^L(A \wedge_{R'}^L B) 
\end{tikzcd}\]
 \end{rmk}

\begin{thm} \label{BrSS}
Let $A$ be a cofibrant commutative $S$-algebra and let $B$ be a connective cofibrant
commutative $A$-algebra. Let $E$ be an $S$-ring  spectrum. Then there is a multiplicative spectral sequence of
the form 
\[   E^2_{n,m} = \THH_n\bigl(B;HE_m^S(B \wedge_A B)\bigr) \Longrightarrow E_{n+m}^S\bigl(\THH(A; B)\bigr).\]
\end{thm}
\begin{proof}
Since $B$ is a cofibrant commutative $S$-algebra, the map 
\[\Gamma^S B \wedge_S \Gamma^S B \to B \wedge_S B\]
is a weak equivalence by \cite[Theorem VII.6.5, Theorem VII.6.7]{EKMM}. By \cite[p.2]{LewMand} we therefore 
have a strongly convergent spectral sequence of the form
\[ E^2_{n,m} = \Tor_{n,m}^{S_*}(B_*, B_*) \Longrightarrow \pi_{n+m}(B  \wedge_S B).\]
This shows that $B \wedge_S B$ is connective.  
Since $B^e$ is a cofibrant in $\mathscr{C}\mathscr{A}_S$, Theorem \ref{AHss} yields a strongly convergent spectral sequence of the form 
\[E^2_{n,m} = \Bigl(H\pi_m\bigl(\Gamma^S E \wedge_S (B \wedge_A B)\bigr)\Bigr)_n^{B^e} B \Longrightarrow \bigl(\Gamma^S E \wedge_S (B \wedge_A B)\bigr)_{n+m}^{B^e} B.\]
By Remark \ref{assocSR} and Proposition \ref{AHssM} the spectral sequence is multiplicative. 
We have 
\begin{align*}
E^2_{n,m} & = \pi_n\Bigl(H\pi_m\bigl(\Gamma^S E \wedge_S (B \wedge_A B)\bigl) \wedge_{B^e}^L B\Bigl) \\
          & = \pi_n\bigr( HE_m^S(B \wedge_A B) \wedge_{B^e}^L B \bigl) \\
          & = \THH_n\bigl(B; HE_m^S(B \wedge_A B)\bigr).
          \end{align*}
By  Remark \ref{assocSR} and by Lemma \ref{THHid} we have isomorphisms 
\begin{align*} 
\bigl(\Gamma^S E \wedge_S (B \wedge_A B)\bigr)_{*}^{B^e} B  & =  \pi_*\Bigl(\bigl(\Gamma^S E \wedge_S (B \wedge_A B)\bigr) \wedge_{B^e}^L B\Bigr) \\
 & \cong \pi_*\Bigl(E \wedge_S^L \bigl((B \wedge_A B) \wedge_{B^e}^L B \bigr)\Bigr) \\
 & \cong \pi_*\bigl(E \wedge_S^L \THH(A;B)\bigr) .
\end{align*}
that are compatible with the multiplications.  
\end{proof}

We want  to find a more computable description of the $E^2$-page in special cases. To see that different ring spectra are isomorphic, the following remark  will be useful. 

\begin{rmk} \label{uniems}
Let $R$ be a connective commutative $S$-algebra.  Let $A$ and $B$ be $R$-ring spectra with $\pi_*(A) = 0 = \pi_*(B)$ for $* \neq 0$. Suppose that there is an isomorphism $A_0 \cong B_0$ that is $R_0$-linear and compatible with the unit and multiplication. 
Then, by using \cite[p.3]{LewMand},  one gets $A \cong B$ as $R$-ring spectra. 
\end{rmk}

Now let $A$ be a connective cofibrant commutative $S$-algebra. Then, we have a map of commutative $S$-algebras $A \to H \pi_0(A)$ realizing the identity on $\pi_0$. We consider  $\THH(A; H\pi_0(A)) = H\pi_0(A) \wedge_{A^e}^L A$. By Remark \ref{uniems}   this does not depend on the choice of the representative of the homotopy class $A \to H\pi_0(A)$ in $\bar{h}\mathscr{C}\mathscr{A}_S$.  
By using the commutative diagram (\ref{postnikovnat}) we also see that $\THH(A; H\pi_0(A))$ is isomorphic  to 
$H\pi_0(A) \wedge_{A^e}^L A$ built via 
$A^e \to H\pi_0(A^e) \to H\pi_0(A)$.

\begin{lem} \label{BrE2}
In the situation of Theorem  \ref{BrSS} we assume that $E_m^S(B \wedge_A B)$ is an $\mathbb{F}_p$-vector space for all $m$ and that $\pi_0(B)/{p\pi_0(B)} = \mathbb{F}_p$ as rings. 
Then we have an isomorphism
\[E^2_{n,m} \cong E_m^S(B \wedge_A B) \otimes_{\mathbb{F}_p} \THH_n(B; H\Fp).\]
that is compatible with the multiplication. 
Here, the multiplication on the right side is given by $a \otimes b \cdot c \otimes d = (-1)^{|b| |c|} a  c \otimes b  d$.  
The $(B, B)$-bimodule structure on $H\mathbb{F}_p$ is given by the morphism of commutative $S$-algebras $B \to H\pi_0(B) \to H\mathbb{F}_p$.
\end{lem}
\begin{proof}
We have \[E^2_{n,m} \cong \pi_n(HE^S_m(B \wedge_A B) \wedge_{B^e} \Gamma^{B^e} B).\]
The $B^e$-action on $HE^S_m(B \wedge_A B)$ is given by pulling back the $H\pi_0(B^e)$-action along the map $B^e \to H\pi_0(B^e)$. 
The K\"unneth map 
\[\pi_0(B) \otimes_{\mathbb{Z}} \pi_0(B) \to \pi_0(B \wedge_S B)\]
is a ring map, and it is an isomorphism because it is the edge homomorphism in the Tor spectral sequence
\[  E^2_{n,m} = \Tor_{n,m}^{S_*}(B_*, B_*) \Longrightarrow \pi_{n+m}(B  \wedge_S B)\]
(see \cite[pp.2--3]{LewMand}) and because $B$ is connective. 
The $\pi_0(B) \otimes \pi_0(B) = \pi_0(B \wedge_S B)$-action on $E_m^S(B \wedge_A B)$ is given by pulling back the $\mathbb{F}_p$-action along 
\[ \pi_0(B) \otimes \pi_0(B) \to \pi_0(B) \to \pi_0(B)/{p\pi_0(B)} = \mathbb{F}_p.\]
This is because the two  actions of $\pi_0(B)$ induced  by  the two inclusions \[\pi_0(B) \to \pi_0(B) \otimes \pi_0(B)\] have to factor  
over  $\pi_0(B)/{p\pi_0(B)}$-actions because $E_m^S(B \wedge_A B)$ is an $\mathbb{F}_p$-vector space,
and because there is only one $\pi_0(B)/{p\pi_0(B)} = \mathbb{F}_p$-action on a given abelian group. 
Because of the uniqueness of Eilenberg-Mac Lane spectra we can write
\[E^2_{n,m} \cong \pi_n(HE^S_m(B \wedge_A B) \wedge_{B^e} \Gamma^{B^e} B), \]
where we realize $HE^S_m(B \wedge_A B)$ as an $H\mathbb{F}_p$-module and equip it with a $B^e$-module structure by pullback along
\begin{equation*} \label{alguc}
 B^e \to H\pi_0(B^e) \to H\pi_0(B) \to H\mathbb{F}_p.
 \end{equation*} 
We can also assume that $HE^S_m(B \wedge_A B)$ is a cell $H\mathbb{F}_p$-module. 
We have 
\[ HE^S_m(B \wedge_A B) \wedge_{B^e} \Gamma^{B^e} B \cong HE^S_m(B \wedge_A B) \wedge_{H\mathbb{F}_p} (H\mathbb{F}_p \wedge_{B^e} \Gamma^{B^e} B).\] 
With the Tor spectral sequence we get
\begin{align*}
E^2_{n,m} & = E_m^S(B \wedge_A B) \otimes_{\mathbb{F}_p} \pi_n(H \mathbb{F}_p \wedge_{B^e} \Gamma^{B^e} B ) \\
          & = E_m^S(B \wedge_A B) \otimes_{\mathbb{F}_p} \pi_n(H \mathbb{F}_p \wedge_{B^e}^L  B ).
\end{align*}

By Remark \ref{uniems}  this is isomorphic to 
\[E_m^S(B \wedge_A B) \otimes_{\mathbb{F}_p} \THH_n(B; H\mathbb{F}_p).\]

It remains to  prove that the identification is  compatible with the multiplication. 

Our spectral sequence is given as an Atiyah-Hirzebruch spectral sequence with 
 $E^1$-page
\[E^1_{n,m} = \pi_{n+m}\bigl(F \wedge_{B^e} (\Gamma^{B^e} B)^n/{(\Gamma^{B^e} B)^{n-1}} \bigr),\] 
where $F = \Gamma^{B^e}\bigl(\Gamma^S E \wedge_S (B \wedge_A B)\bigr)$. 
The multiplication on the $E^1$-page is given by the maps 
\begin{equation} \label{multi1}
 F \wedge_{B^e} F \to F
 \end{equation}
 and \begin{equation}  \label{multi2}
\begin{tikzcd}  (\Gamma^{B^e} B)^n/{(\Gamma^{B^e} B)^{n-1}} \wedge_{B^e} (\Gamma^{B^e} B)^{n'}/{(\Gamma^{B^e} B)^{n'-1}} \ar{r} & (\Gamma^{B^e} B)^{n+n'}/{(\Gamma^{B^e} B)^{n+n'-1}}.
 \end{tikzcd}
 \end{equation}
 Here, the maps (\ref{multi2}) comes from a cellular representative 
 \begin{equation} \label{multi3}
 \Gamma^{B^e}B \wedge_{B^e} \Gamma^{B^e}B \to \Gamma^{B^e}B 
 \end{equation}
 of the product $B \wedge_{B^e}^L B \to B$. 
One easily sees that the isomorphism between 
\[E_m^S(B \wedge_A B) \otimes_{\mathbb{F}_p} \THH_n(B; H\mathbb{F}_p) \cong F_m \otimes_{\mathbb{F}_p} \pi_n(H \mathbb{F}_p \wedge_{B^e} \Gamma^{B^e} B )\] and the homology of the $E^1$-page is given as follows: 
For 
\[x \in F_m \otimes_{\mathbb{F}_p} \pi_n(H \mathbb{F}_p \wedge_{B^e} \Gamma^{B^e} B )\] choose a preimage 
$y$ under
\[\begin{tikzcd}
 F_m \otimes_{\mathbb{F}_p} \pi_n\bigl(H \mathbb{F}_p \wedge_{B^e} (\Gamma^{B^e} B)^n \bigr)  \ar[two heads]{r} & F_m \otimes_{\mathbb{F}_p} \pi_n\bigl(H \mathbb{F}_p \wedge_{B^e} \Gamma^{B^e} B \bigr).
 \end{tikzcd}\]
 Then $x$ is mapped to the homology class of the image of $y$ under  the composition of 
 \[ \begin{tikzcd}
  F_m \otimes_{\mathbb{F}_p} \pi_n\bigl(H \mathbb{F}_p \wedge_{B^e} (\Gamma^{B^e} B)^n \bigr) \ar{r} & F_m \otimes_{\mathbb{F}_p} \pi_n\bigl(H \mathbb{F}_p \wedge_{B^e} (\Gamma^{B^e} B)^n/{(\Gamma^{B^e} B)^{n-1}} \bigr)
  \end{tikzcd}\]
 and 
  \begin{eqnarray*} 
 & &  F_m \otimes_{\mathbb{F}_p} \pi_n\bigl(H \mathbb{F}_p \wedge_{B^e} (\Gamma^{B^e} B)^n/{(\Gamma^{B^e} B)^{n-1}} \bigr) \\
&   \cong  & F_m \otimes_{\mathbb{F}_p} \mathbb{F}_p \otimes_{(B^e)_0}\pi_n\bigl((\Gamma^{B^e} B)^n/{(\Gamma^{B^e} B)^{n-1}}\bigr)  \\
&  \cong &  F_m  \otimes_{(B^e)_0}\pi_n\bigl((\Gamma^{B^e} B)^n/{(\Gamma^{B^e} B)^{n-1}}\bigr)  \\
& \cong &   \pi_{n+m}\bigl(F \wedge_{B^e} (\Gamma^{B^e} B)^n/{(\Gamma^{B^e} B)^{n-1}} \bigr). 
 \end{eqnarray*}
Every intermediate graded abelian group appearing in this identification can be endowed with a multiplication by using 
the maps $H\mathbb{F}_p \wedge_{B^e} H\mathbb{F}_p \to H\mathbb{F}_p$, (\ref{multi1}), (\ref{multi2}), (\ref{multi3})  and 
 \begin{equation*}
 \begin{tikzcd}
 (\Gamma^{B^e} B)^n \wedge_{B^e} (\Gamma^{B^e} B)^{n'} \to (\Gamma^{B^e} B)^{n+n'}. 
 \end{tikzcd}
 \end{equation*}
For multiplications on tensor products we use the sign convention given in the statement of the lemma. Then, every map in the identification is compatible with the multiplication. 
This proves the lemma. 
\end{proof}

\begin{lem} \label{nat}
Suppose that we have a commutative diagram in $\mathscr{C}\mathscr{A}_S$
\[\begin{tikzcd}
A \ar[tail]{r}  \ar{d} & B \ar{d} \\
C \ar[tail]{r} & D,
\end{tikzcd}\] 
where $A$ and $C$ are cofibrant commutative $S$-algebras, $B$ and $D$ are connective and where  the horizontal maps are cofibrations in $\mathscr{C}\mathscr{A}_S$. 
Furthermore, suppose that we have a morphism of $S$-ring spectra $E \to F$.  
Then we get a morphism of Brun spectral sequences converging to the map  $E_*^S\THH(A;B) \to F^S_*\THH(C;D)$.   
If additionally $E^S_m(B \wedge_A B)$ and $F^S_m(D \wedge_C D)$ are $\mathbb{F}_p$-vector spaces  and $\pi_0(B)/{p\pi_0(B)} = \mathbb{F}_p = \pi_0(D)/{p\pi_0(D)} $  as rings, the map on the $E^2$-pages identifies with the map 
\[\begin{tikzcd}
E^S_*(B \wedge_A B) \otimes_{\mathbb{F}_p} \THH_*(B; H\mathbb{F}_p) \ar{r} & F^S_*(D \wedge_C D) \otimes_{\mathbb{F}_p} \THH_*(D; H\mathbb{F}_p)
\end{tikzcd}\]
that is given by the obvious map on the first tensor factor and that is induced by the map 
\[\begin{tikzcd}
\THH(B; H\mathbb{F}_p) \ar{r}{\cong} & H\mathbb{F}_p \wedge_{B^e}^L B \ar{r} & H\mathbb{F}_p \wedge_{B^e} ^LD \ar{r} & H\mathbb{F}_p \wedge_{D^e}^L D
\end{tikzcd}\]
on the second tensor factor. 
Here, the $B^e$-module action on $H\mathbb{F}_p$  in the second smash product is given by 
\[\begin{tikzcd}
B^e \ar{r} & D^e \ar{r} & D \ar{r} & H \pi_0(D) \ar{r} & H\mathbb{F}_p 
\end{tikzcd}\]
and the isomorphism is defined using Remark \ref{uniems}. 
\end{lem}
\begin{proof}
By Remarks \ref{natTHH}, \ref{natTHHid}  and \ref{assocSR} the morphism 
$E_*^S\THH(A;B) \to F_*^S\THH(C;D)$  corresponds to the composition of
\[\begin{tikzcd}
\bigl(\Gamma^SE \wedge_S(B \wedge_A B)\bigr) \wedge_{B^e}^L B \ar{r} & \bigl(\Gamma^SF \wedge_S(D \wedge_CD)\bigr) \wedge_{B^e}^L D  \end{tikzcd}\]
and 
\[\begin{tikzcd} 
\bigl(\Gamma^SF \wedge_S(D \wedge_CD)\bigr) \wedge_{B^e}^L D\ar{r} & \bigl(\Gamma^S F \wedge_S(D \wedge_CD)\bigr) \wedge_{D^e}^L D.  
\end{tikzcd}\]
By Lemma  \ref{natati} and Lemma \ref{natatichanger}  we have Atiyah-Hirzebruch spectral sequences converging to these maps. 

Now, assume that the additional condtitions are satisfied. We identify the map on $E^2$-pages. Recall from the proof of Lemma \ref{BrE2} that we identified 
\[
HE_m^S(B \wedge_A B) \wedge_{B^e} \Gamma^{B^e}B 
\]
with 
\[ HE^S_m(B \wedge_A B) \wedge_{H\mathbb{F}_p} H\mathbb{F}_p \wedge_{B^e} \Gamma^{B^e} B, \] 
where in the second smash  product $HE_m(B \wedge_A B)$ is realized as cell $H\mathbb{F}_p$-module and the $B^e$-action on $H\mathbb{F}_p$ is given by 
\begin{equation} \label{act1}
\begin{tikzcd}
B^e \ar{r} & H\pi_0(B^e)  \ar{r} & H \pi_0(B)  \ar{r} & H\mathbb{F}_p. 
\end{tikzcd}
\end{equation}
The analogous statement holds for $E$, $A$ and $B$ replaced by $F$, $C$ and $D$. 
We have a cell approximation $\Gamma^1 \xrightarrow{\sim} H\mathbb{F}_p$ of $H\mathbb{F}_p$   as $H\mathbb{F}_p \wedge_S B^e$-module, where the $B^e$-module action is given by 
(\ref{act1}). We also have a cell approximation $\Gamma^2 \xrightarrow{\sim} H\mathbb{F}_p$ of $H\mathbb{F}_p$   as $H\mathbb{F}_p \wedge_S B^e$-module, where the $B^e$-module action is given by 
\[\begin{tikzcd}
B^e \ar{r} & D^e \ar{r} & H\pi_0(D^e) \ar{r} & H\pi_0(D) \ar{r} & H \mathbb{F}_p.
\end{tikzcd}\]
By fact (\ref{uniquems}), $H\mathbb{F}_p$ equipped with the first $H\mathbb{F}_p \wedge_S B^e$-module structure is isomorphic in $\mathscr{D}_{H\mathbb{F}_p \wedge_S B^e}$  to $H\mathbb{F}_p$ equipped with the second module structure. The isomorphism is represented by a map $\Gamma^1 \to \Gamma^2$.  By (\ref{uniquems}) we have a map between the cell $H\mathbb{F}_p$-modules $HE^S_m(B \wedge_A B)$ and $HF_m^S(D \wedge_C D)$ inducing the obvious map in homotopy.  The maps 
$H\mathbb{F}_p  \xleftarrow{} \Gamma^1 \to \Gamma^2 \to H\mathbb{F}_p$, $HE^S_m(B \wedge_A B) \to HF^S_m(D \wedge_C D)$ and $\Gamma^{B^e} B \to \Gamma^{B^e} D$ induce a map $(f_1^m)_*$
\[\begin{tikzcd}
 \pi_*\bigl(HE^S_m(B \wedge_A B) \wedge_{H\mathbb{F}_p} H\mathbb{F}_p \wedge_{B^e} \Gamma^{B^e} B\bigr) \ar{r} & \pi_*\bigl(HF^S_m(D \wedge_C D) \wedge_{H\mathbb{F}_p} H\mathbb{F}_p \wedge_{B^e} \Gamma^{B^e}D\bigr).
\end{tikzcd}\]
We also have a map $(f_2^m)_*$
\[\begin{tikzcd}
\pi_*\bigl(HF^S_m(D \wedge_C D) \wedge_{H\mathbb{F}_p} H\mathbb{F}_p \wedge_{B^e} \Gamma^{B^e}D\bigr) \ar{r} & \pi_*\bigl( HF^S_m(D \wedge_C D) \wedge_{H\mathbb{F}_p} H\mathbb{F}_p \wedge_{D^e} \Gamma^{D^e} D \bigr),
\end{tikzcd}\]
induced by a map $\Gamma^{B^e}D \to \Gamma^{D^e}D$   in $\mathscr{M}_{B^e}$ lifting the identity up to homotopy. 
Using (\ref{uniquems}) it follows that the map on $E^2$-pages is given by the maps $(f_2^m)_* \circ (f_1^m)_*$. 
In the proof of Lemma \ref{BrE2} we also did the identifications 
\begin{eqnarray*}
 & & \pi_*\bigl(HE^S_m(B \wedge_A B) \wedge_{H\mathbb{F}_p} H\mathbb{F}_p \wedge_{B^e} \Gamma^{B^e} B\bigr) \\
& \cong & E_m^S(B \wedge_A B) \otimes_{\mathbb{F}_p} \pi_*(H \mathbb{F}_p  \wedge_{B^e} \Gamma^{B^e} B )\\
& \cong & E^S_m(B \wedge_A B) \otimes_{\mathbb{F}_p} \pi_*(H \mathbb{F}_p  \wedge_{B^e}^L B) \\
& \cong & E_m^S(B \wedge_A B) \otimes_{\mathbb{F}_p} \THH_*(B; H\mathbb{F}_p) 
\end{eqnarray*}
and analogously for  $F$, $C$, and $D$.  Again by using (\ref{uniquems}), one sees  that under these isomorphisms the maps $(f_2^m) \circ (f_1^m)_*$  corresponds to map in the statement of the lemma. 
\end{proof}

\begin{rmk}
One easily sees that the Brun spectral sequence is functorial with respect to morphisms of $S$-ring spectra and the vertical maps in diagrams as in Lemma \ref{nat}. 
\end{rmk}


\section{The topological Hochschild homology of \texorpdfstring{$\ku$}{ku}} \label{kusec}
In \cite{Au} Ausoni computed the mod $(p, v_1)$ homotopy of the topological Hochschild homology of $p$-completed connective complex $K$-theory for an odd prime $p$. Ausoni used the B\"okstedt spectral sequence. In this chapter we present a shorter computation for $p \geq 5$ using the Brun spectral sequence. 

We take for the ring spectrum $E$ in the Brun spectral sequence  the mod $p$ Moore spectrum and the mod $(p,v_1)$ Toda-Smith  complex. 
Recall the following:

\begin{rmk}
Let $p$ be a prime with $p \geq 5$. Let $V(0)$ be the mod $p$ Moore spectrum and let $V(1)$ be the mod $(p,v_1)$ Toda-Smith complex. We can suppose that $V(0)$ and $V(1)$ are cell $S$-modules.   
Because of $p \geq 5$ we have the following properties (see \cite{OkaMoore},\cite{OkaMult} and \cite{Okafewcells}):
\begin{itemize}
\item The $S$-modules $V(0)$ and $V(1)$ are associative and commutative $S$-ring spectra. 
\item We have distinguished triangles in $\mathscr{D}_S$
      \begin{equation} \label{dist1}
      \begin{tikzcd}
 S \ar{r}{p \cdot \id} & S \ar{r} & V(0) \to  \Sigma  S 
 \end{tikzcd}
 \end{equation}
 and 
 \begin{equation*} \label{dist2}
 \begin{tikzcd}
 \Sigma^{2p-2}V(0) \ar{r} & V(0) \ar{r} & V(1) \ar{r} & \Sigma^{2p-1} V(0). 
\end{tikzcd}
\end{equation*} 
The maps $S \to V(0)$ and $V(0) \to V(1)$ are maps of $S$-ring spectra. 
\end{itemize}
\end{rmk}
Let $p$ be a prime with $p \geq 5$. 
Let $\ku$ be $p$-completed connective complex $\K$-theory. 
Following Ausoni we take as a commutative $S$-algebra model for $\ku$ the $p$-completion of the algebraic $K$-theory of $\cup_{i \geq 0}\mathbb{F}_{l^{p^i(p-1)}}$, $\K(\cup \mathbb{F}_{l^{p^i(p-1)}})_p$,  where $l$ is a prime that generates $(\mathbb{Z}/{p^2})^*$. 
By \cite[Example VI.5.2]{MayRingSp},\cite{MayBip}, \cite{Maywhat} and \cite[Corollary II.3.6]{EKMM}  algebraic $K$-theory can be realized as a functor from commutative rings to commutative $S$-algebras. By $p$-completion we mean Bousfield localization with respect to $V(0)$ as defined in \cite[Section VIII.1]{EKMM}.  
By the proofs of \cite[Lemma VII. 5.8]{EKMM}, \cite[Lemma VII.5.2]{EKMM}  and \cite[Theorem VIII. 2.2]{EKMM}
 the $p$-completion of a commutative $S$-algebra can be constructed as a commutative $S$-algebra in such a way that we get a functor $(-)_p: \mathscr{C}\mathscr{A}_S \to \mathscr{C}\mathscr{A}_S$ with values in cofibrant commutative $S$-algebras.

We have 
\[\pi_*(\ku) = \mathbb{Z}_p[u]\]
as rings, where $|u| = 2$.  Using the cofibre sequence (\ref{dist1}) one gets \[  V(0)_*^S\ku = P(u), \] 
 where $P(-)$ denotes the polynomial algebra over $\mathbb{F}_p$. By  \cite{Au} we have 
\[ V(1)_*^S\ku = P_{p-1}(u).\]  Here, $P_{p-1}(-) = P(u)/{(u^{p-1})}$ is the truncated polynomial algebra over $\mathbb{F}_p$. The map $V(0)_*^S\ku \to V(1)_*^S\ku$ is the quotient map. 
By \cite[Theorem 2.5]{Au} we have 
\[{H\mathbb{F}_p}_*^S\ku = P_{p-1}(u) \otimes P(\bar{\xi}_1, \bar{\xi}_2, \dots) \otimes E(\bar{\tau}_2, \dots),\]
where the tensor product is taken over $\mathbb{F}_p$ and where $E(-)$ denotes the exterior algebra over $\mathbb{F}_p$. We have  $|\bar{\xi}_i| = 2p^i-2$ and $|\bar{\tau}_i| = 2p^i-1$, and  $u$ is the image of $u \in \pi_2(\ku)$ under the Hurewicz map
\[\begin{tikzcd}
  \pi_*(\ku) \ar{r}{\cong} & \pi_*(S \wedge_S^L \ku) \ar{r} & \pi_*(H\mathbb{F}_p \wedge_S^L \ku). 
  \end{tikzcd} \] 
 Recall that $\ell := \K(\cup \mathbb{F}_{l^{p^i}})_p$ is a commutative $S$-algebra model for the $p$-completed connective Adams sumand (\cite[Proposition 9.2]{McSt}, \cite[Section 2]{BakRichUni}).  The map $\pi_*(\ell) \to \pi_*(\ku)$ induced by the inclusion of fields identifies $\pi_*(\ell)$ with the subring $\mathbb{Z}_p[u^{p-1}]$ of $\pi_*(\ku)$. 
  We have $V(0)_*^S\ell = P(u^{p-1})$ and $V(1)_*^S\ell = \mathbb{F}_p$.

In order to calculate $V(1)_*^S\THH(\ku)$ we proceed as follows: 

We factor the map $\ku \to H\mathbb{Z}_p$ in the category of commutative $S$-algebras as 
\[\begin{tikzcd}
\ku  \ar[tail]{r} & \hat{H} \mathbb{Z}_p \ar[two heads]{r}{\sim} & H\mathbb{Z}_p. 
\end{tikzcd}\]
To simplify the notation we will write $H\mathbb{Z}_p$ for $\hat{H}\mathbb{Z}_p$  most of the time. 

We have a Brun spectral sequence 
\begin{equation} \label{brunku1}
 E^2_{*,*} = \THH_*(H\Zp; HV(0)_*^S(H\Zp \wedge_{\ku} H\Zp))  \Longrightarrow V(0)_*^S\THH(\ku; H\Zp) 
 \end{equation} 
and by Lemma \ref{BrE2} the $E^2$-page can be identified with 
\[E^2_{*,*} =  V(0)_*^S(H\Zp \wedge_{\ku} H\Zp) \otimes \THH_*(H\Zp; H\Fp).\]
We compute this spectral sequence in Subsection \ref{kurel}. 
In Subsection \ref{kuabs} we compute the Brun spectral sequence 
 \begin{equation} \label{brunku2}
 E^2_{*,*} = \THH_*(\ku; HV(1)_*^S\ku) \Longrightarrow V(1)_*^S\THH(\ku).
 \end{equation}
Again by Lemma \ref{BrE2}, the $E^2$-page can be identified with 
\begin{eqnarray*}
 E^2_{*,*}  =   V(1)_*^S\ku \otimes \THH_*(\ku; H\Fp) \cong   P_{p-1}(u) \otimes \THH_*(\ku; H\F_p).  
 \end{eqnarray*} 
 Note that by Remark \ref{assocSR} we have an isomorphism of $S$-ring spectra 
 \begin{eqnarray*}
 V(0) \wedge_S^L \THH(\ku; H\mathbb{Z}_p) & = & V(0) \wedge_S^L (H\mathbb{Z}_p \wedge_{\ku^e}^L \ku)  \\
    & \cong & (V(0) \wedge_S H \mathbb{Z}_p) \wedge_{\ku^e}^L \ku 
 \end{eqnarray*}
Since the $H\mathbb{Z}_p$-ring spectra $V(0) \wedge_S H\mathbb{Z}_p$ and $H\mathbb{F}_p$ are isomorphic by Remark \ref{uniems}, this is isomorphic to $\THH(\ku; H\mathbb{F}_p)$. 
Thus,  the abutment of the spectral sequence (\ref{brunku1})  
computes the input of the spectral sequence  (\ref{brunku2}).

\begin{nota}
An infinite cycle in a spectral sequence is a class $b$ such that we have $d^s(b) = 0$ for all $s$. 
A permanent cycle  is an infinite cycle  that is not in the image of  $d^s$ for any $s$.
We write $b \doteq c$ if $b$ and $c$ are equal up to multiplication by a unit in $\mathbb{F}_p$.
\end{nota}
\subsection{The mod \texorpdfstring{$p$}{p} homotopy of \texorpdfstring{$\THH(\ku; H\mathbb{Z}_{p})$}{THH(ku,HZp)}} \label{kurel}

In this subsection we calculate the $\F_p$-algebra $V(0)_*^S\THH_*(\ku, H\mathbb{Z}_p)$ using the spectral sequence (\ref{brunku1}). 
By B\"okstedt's computations \cite{Bo2}, we have 
\[\THH_*(H\Zp; H\F_p) = E(\lambda_1) \otimes P(\mu_1)\]
 with $|\lambda_1| = 2p-1$ and $|\mu_1| = 2p$.
\begin{lem} \label{tensorku} 
We have $V(0)_*^S(H\Zp \wedge_{\ku} H\Zp) = E(\sigma u)$ with $|\sigma u| = 3$.
\end{lem}
\begin{proof}
Since $H\mathbb{Z}_p$ is a cofibrant commutative $\ku$-algebra, we have
\[ V(0) \wedge_S^L(H\mathbb{Z}_p \wedge_{\ku} H\mathbb{Z}_p) \cong V(0) \wedge_S^L(H\mathbb{Z}_p \wedge_{\ku}^L H\mathbb{Z}_p)\] as $S$-ring spectra by Lemma \ref{cof1en}.  
Using Remark \ref{assocSR} one gets that this is isomorphic  to the $S$-ring spectrum $H\mathbb{F}_p \wedge_{\ku}^L H\mathbb{Z}_p$. 

By \cite{LewMand} we have a strongly convergent spectral sequence of the form
\[ E^2_{n,m} = \Tor^{{\ku}_*}_{n,m}(\Fp, \Zp) \Longrightarrow \pi_{n+m}(H\Fp \wedge_{\ku}^L H\Zp).   \]
We have the following free resolution of $\Zp$ as a ${\ku}_*$-module:
\[0 \to \Sigma^2{\ku}_* \xrightarrow{u} {\ku}_* \to \Zp \to 0.\]
Thus, the $E^2$-page is $\Fp$ in the bidegrees $(0,0)$ and $(1,2)$ and zero in all other bidegrees. 
There cannot be any differentials. Since there is only one graded-commutative $\Fp$-algebra that has  $\Fp \oplus \Sigma^3\Fp$ as an underlying $\mathbb{F}_p$-vector space, this proves the lemma. 

\end{proof}
Thus, the $E^2$-page of the spectral sequence (\ref{brunku1}) is 
\[E(\sigma u) \otimes E(\lambda_1) \otimes P(\mu_1), \]
where $\sigma u$, $\lambda_1$ and $\mu_1$ have the bidegrees  $(0, 3)$,  $(2p-1,0)$ and  $(2p,0)$.

\begin{lem}
The spectral sequence (\ref{brunku1})  collapses at the $E^2$-page. 
\end{lem}
\begin{proof}
The class $\sigma u$ is an infinite cycle, because it lies in column zero. The class $\lambda_1$ is an infinite cycle, because the $E^2$-page is zero in total degree $2p-2$. The class $\mu_1$ is also an infinite cycle: In total degree  $2p-1$ the 
$E^2$-page is given by $\Fp\{\lambda_1\}$. The differentials of $\mu_1$ lies in columns $\leq 2p-2$ and $\lambda_1 $ lies in 
 column $2p-1$. Thus, the differentials of $\mu_1$ cannot hit $\lambda_1$. 
\end{proof}

\begin{lem} \label{kuBR1}
We have 
\[ V(0)_*^S\THH(\ku; H\mathbb{Z}_p) = E(\sigma u, \lambda_1) \otimes P(\mu_1)\]
with 
$|\sigma u| = 3$, $|\lambda_1| = 2p-1$ and $|\mu_1| = 2p$. 
\end{lem}
\begin{proof}
Since the $E^{\infty} = E^2$-page is a free graded-commutative $\mathbb{F}_p$-algebra, there cannot be any multiplicative extensions. 
\end{proof}

\subsection{The \texorpdfstring{$V(1)$}{V(1)}-homotopy of \texorpdfstring{$\THH(\ku)$}{THH(ku)}} \label{kuabs}
In this subsection we calculate $V(1)_*\THH(\ku)$ via  the spectral sequence (\ref{brunku2}). 
By Lemma \ref{kuBR1} we have 
\[ E^2_{*,*} = P_{p-1}(u) \otimes  E(\sigma u, \lambda_1) \otimes P(\mu_1),    \]
where $u$, $\sigma u$, $\lambda_1$ and $\mu_1$ have  bidegrees $(0,2)$, $(3,0)$, $(2p-1,0)$ and $(2p,0)$ respectively. 

\begin{lem} \label{kudiffhilf}
There are  non-trivial elements $u' \in V(1)_2^S\THH(\ku)$ and $\sigma u' \in V(1)_3^S\THH(\ku)$ satisfying the equation 
\[ (u')^{p-2} \sigma u' = 0.\]
\end{lem}
\begin{proof}
To see this, we pass to homology. We have $u^{p-1} = 0$ in $\HFp_*^S\ku$. 
Recall from \cite[p.1238, Proposition 5.10]{AnRo} that we have a map $\ku \wedge S^1 \to \THH(\ku)$ such that the induced map in homology  $\sigma_*: {H\mathbb{F}_p}_*^S\ku \to {H\mathbb{F}_p}_{*+1}^S\THH(\ku)$ is a derivation. We apply  $\sigma_*$ to the equation $u^{p-1} = 0$  and get 
\[u^{p-2}\sigma_*(u) = 0\]  in $\HFp_*^S\THH(\ku)$, where $u$ is the image of $u$ under the unit ${H\mathbb{F}_p}_*^S\ku \to {H\mathbb{F}_p}_*^S\THH(\ku)$. Since the unit has a left inverse,  $u \in (H\mathbb{F}_p)_2^S\THH(\ku)$ is not zero. 
 We claim that  the class 
$\sigma_*(u)$  is also not zero. To see this we consider the B\"okstedt spectral sequence \cite[Section 4]{AnRo}
\[ E^2_{*,*}(\ku) = \mathbb{H}^{\mathbb{F}_p}_{*,*}({H\mathbb{F}_p}_*^S\ku) \Longrightarrow {H\mathbb{F}_p}_*^S\THH(\ku).\]
Here, $\mathbb{H}^{\mathbb{F}_p}_{*,*}({H\mathbb{F}_p}_*^S\ku)$ denotes the Hochschild homology of ${H\mathbb{F}_p}_*^S\ku$ over the ground ring  $\mathbb{F}_p$. 
By \cite[Proposition 3.2]{McSt} the class $\sigma_*u$ is a representative of the class $\sigma (u) \in E^{\infty}_{1,2}({\ku})$ that is given in the Hochschild complex by $1 \otimes u$.   The class $\sigma(u) \in E^{\infty}_{1,2}(\ku)$  is not zero, because by  standard computations of Hochschild homology (see \cite[Proposition 3.2, Proposition 3,3]{Au}) the bigraded $\mathbb{F}_p$-vector space $E^2_{*,*}(\ku)$ does not contain any class of total degree $4$ in a filtration degree $\geq 3$, and therefore $\sigma (u) \in E^2_{1,2}(\ku)$ cannot be hit by any differential.  

Recall that the mod $p$ homology of an $S$-module admits a natural comodule structure over the dual Steenrod algebra $A_* = {H\mathbb{F}_p}_*^SH\mathbb{F}_p$ \cite[Theorem 1.1]{BakLaz}. By \cite[Corollary 17.4]{Swi} the image of the Hurewicz map is contained in the $\mathbb{F}_p$-subspace of $A_*$-comodule primitives. Thus, the class $u \in {H\mathbb{F}_p}_*^S\ku$ as well as 
its image under the unit ${H\mathbb{F}_p}_*^S\ku \to {H\mathbb{F}_p}_*^S\THH(\ku)$ are  $A_*$-comodule primitives. 
Since the suspension isomorphism ${H\mathbb{F}_p}_*^S(-) \to {H\mathbb{F}_p}_{*+1}^S(- \wedge S^1)$ is compatible with the comodule action, it follows that the class $\sigma_*(u)$ is  also primitive.

We consider the diagram
\[\begin{tikzcd}
{H\mathbb{F}_p}_*^S\THH(\ku) \ar{r}{f} & {H\mathbb{F}_p}_*^S\bigl(V(1) \wedge_S^L \THH(\ku)\bigr)  & V(1)_*^S\THH(\ku) \ar{l}[swap]{h},
\end{tikzcd}\]
where the right map is the Hurewicz morphism. 
The map $f$ is injective because composed with 
 the K\"unneth isomorphism 
 \[ {H\mathbb{F}_p}_*^S(V(1) \wedge_S^L \THH(\ku)) \cong {H\mathbb{F}_p}_*^SV(1) \otimes {H\mathbb{F}_p}_*^S\THH(\ku)\] 
 it  is the inclusion 
 \[\HFp_*^S\THH(\ku) \to {H\mathbb{F}_p}_*^SV(1) \otimes  {H\mathbb{F}_p}_*^S\THH(\ku); \;\;\; a \mapsto 1 \otimes a.\]
The $S$-module $V(1) \wedge_S^L \THH(\ku)$ is a module (in $\mathscr{D}_S$) over the $
S$-ring spectrum $V(1) \wedge_S^L \ell \cong H\mathbb{F}_p$. Therefore, it is isomorphic in $\mathscr{D}_S$ to a coproduct of  $S$-modules of the form $H\mathbb{F}_p \wedge_S^L S^n_S$ and the Hurewicz morphism $h$ 
induces an isomorphism  onto  the $A_*$-comodule primitives of $\HFp_*^S(V(1) \wedge_S^L \THH(\ku))$. 
We  obtain two non-trivial elements $u'$ and $\sigma_*(u)'$ in $V(1)_*\THH(\ku)$ satisfying the equation 
\[ (u')^{p-2} \sigma_* (u)' = 0.\] 
\end{proof}

\begin{lem}
The classes $u$, $\sigma u$ and $\lambda_1$ in the Brun  spectral sequence (\ref{brunku2}) are infinite cycles. We have $d^s(\mu_1) = 0$ for $2 \leq s \leq 2p-4$ and there is a differential 
\[ d^{2p-3}(\mu_1) \doteq u^{p-2}\sigma u.\]
The spectral sequence collapses at the $E^{2p-2}$-page. 
\end{lem}
\begin{proof}
The class $u$ is an infinite cycle because it lies in  column zero. It is also a permanent 
cycle because  the 
 unit \[V(1) _*^S\ku \to V(1)_*^S\THH(\ku)\] is injective and 
so $V(1)_2^S\THH(\ku)$ is not zero. 
 This implies that $\sigma u$ is an infinite cycle, because the only possible non-trivial differential on $\sigma u$ is $d^{3}(\sigma u) \doteq u$. The class $\lambda_1$ is an infinite cycle because the $E^2$-page is zero in total degree $2p-2$. 

There is a differential $d^{2p-3}(\mu_1) \doteq u^{p-2}\sigma u$: By Lemma \ref{kudiffhilf} we have  two non-trivial elements $u'$ and $\sigma_*(u)'$ in $V(1)_*\THH(\ku)$ satisfying the equation 
\[ (u')^{p-2} \sigma_* (u)' = 0.\] 
Since the $E^{\infty}$-page of the Brun spectral sequence is $\mathbb{F}_p\{u\}$ in total degree two and $\mathbb{F}_p\{\sigma u\}$ in total degree three, up to a unit, $u'$ has to be a representative of $u$  and $\sigma u'$ has to be a representative of $\sigma u$. 
Therefore, we have $u^{p-2}\sigma u  = 0$ in  $E^{\infty}_{*,*}$ and the class $u^{p-2}\sigma u \in E^{2}_{*,*}$ 
has to be hit by a differential. The only class in total degree $2p$ is $\mu_1$. Thus, we have $d^s(\mu_1) = 0$ for $2 \leq s \leq 2p-4$ and $d^{2p-3}(\mu_1) \doteq u^{p-2}\sigma u$. 

The spectral sequence is  concentrated in the rows $0$ to $2p-4$, therefore it has to collapse at the $E^{2p-2}$-page. 
\end{proof}
\begin{lem} \label{kuresult}
We have $E^{\infty}_{*,*} = E^{2p-2}_{*,*} =  \Omega^{\infty}_{*,*} \otimes E(\lambda_1)$, where $\Omega^{\infty}_{*,*}$ is the bigraded commutative $\mathbb{F}_p$-algebra with generators
\[u, \,\,\, \mu_2, \,\,\, a_0, \dots, a_{p-1}, \,\,\, b_1, \dots, b_{p-1}\]
in bidegrees $|u| = (0,2)$, $|\mu_2| = (2p^2,0)$, $|a_i| = (2pi+3,0)$, $|b_i| = (2pi, 2)$ and relations
\begin{alignat}{3}
& ~~~ u^{p-1}  & = ~&  0, &&  \label{rel1}\\
& u^{p-2}a_i ~ & = ~&  0 && \text{~~for~}  0 \leq i \leq p-2, \label{rel2} \\
& u^{p-2}b_i ~  & = ~& 0  && \text{~~for~}  1 \leq i \leq p-1, \label{rel3} \\
& ~~~ b_ib_j & = ~&   ub_{i+j} && \text{~~for~}   i+j \leq p-1, \label{rel4} \\
& ~~~ a_i b_j & = ~&  u a_{i+j} && \text{~~for~}  i+j \leq p-1, \label{rel5} \\
& ~~~ b_ib_j  & = ~&  u b_{i+j-p} \mu_2 && \text{~~for~}  i+j \geq p,
\label{rel6} \\
& ~~~ a_ib_j   &  = ~& u a_{i+j-p} \mu_2 &&\text{~~for~}  i+j \geq p,
\label{rel7}  \\
& ~~~ a_i a_j  & = ~&  0  &&\text{~~for~} 0 \leq i,j \leq p-1. \label{rel8}
\end{alignat}
Here, we use the convention $b_0 = u$.  By bigraded commutative $\mathbb{F}_p$-algebra we mean a bigraded $\mathbb{F}_p$-algebra which is commutative in the following sense: 
If $a$ and $b$ are elements of total degree $n$ and $m$, we have $ab = (-1)^{nm}ba$. 
\end{lem}
\begin{proof}
We have $E^{2p-2}_{*,*} = H_*\bigl((\Omega_{*,*}, d')\bigr) \otimes E(\lambda_1)$, where 
\[\Omega_{*,*} = P_{p-1}(u) \otimes E(\sigma u) \otimes P(\mu_1)\]
and $d'$ is the restriction of the differential of $E^{2p-3}$. 
Computing homology, we find that as an $\mathbb{F}_p$-vector space 
 \begin{eqnarray} \label{basis}
 H_*\bigl((\Omega_{*,*}, d')\bigr) = &&  \bigoplus_{n \geq 0} \Fp\{\mu_1^{pn}\} \oplus \bigoplus_{\substack{n\geq 0\\i= 1, \dotsc,p-2}} \F_p\{u^i  \mu_1^n\} 
 \oplus   \bigoplus_{\substack{n \geq 0\\i = 0, \dotsc p-3}} \Fp\{u^i  \sigma u  \mu_1^n\} \nonumber \\ 
 & \oplus &
\bigoplus_{n \geq 0} \Fp\{u^{p-2}  \sigma u  \mu_1^{pn+p-1}\}.    
 \end{eqnarray}
 We define classes in $H_*\bigl((\Omega_{*,*}, d')\bigr)$ by $\mu_2 \coloneqq \mu_1^p$, $a_i \coloneqq \sigma u \mu_1^i$ for $i = 0, \dots, p-1$ and $b_i \coloneqq u \mu_1^i$ for $i = 1, \dots, p-1$. Obviously, these classes satisfy the relations (\ref{rel1})-(\ref{rel8}). 
 Note that $u^{p-2}\sigma u \mu_1^{p-1}$ is not in the image of $d'$, so that $u^{p-2}a_i = 0$  only holds for $i = 0, \dots, p-2$. We get a map of bigraded $\mathbb{F}_p$-algebras 
 \[\begin{tikzcd}
  \Omega^{\infty}_{*,*} \ar{r}{f} & H_*\bigl((\Omega_{*,*}, d')).
  \end{tikzcd}\] 
Using the relations (\ref{rel1})-(\ref{rel8}) one easily sees that $f$ maps a generating set bijectively onto the basis of $H_*\bigl((\Omega_{*,*}, d')\bigr)$ given in (\ref{basis}). It follows that  that $f$ is an isomorphism. 
\end{proof}

\begin{thm}
We have 
\[ V(1)_*^S\THH(\ku)  = E(\lambda_1) \otimes \Omega^{\infty}_{*,*},\]
where we consider $\Omega^{\infty}_{*,*}$ as graded $\mathbb{F}_p$-algebra via  the total grading. 
\end{thm}
\begin{proof} 
We use a method from \cite{Au}: 
Recall that we have chosen $\K(\cup \mathbb{F}_{l^{p^i(p-1)}})_p$ as a model for $\ku$, where $l$ is a prime that generates $(\mathbb{Z}/{p^2})^*$.  The Galois group of $\cup \mathbb{F}_{l^{p^i(p-1)}}$ over $\mathbb{F}_l$ is given by $\mathbb{Z}_p \times \mathbb{Z}/{p-1}$.  Let $\delta$ be a generator of the supgroup $\mathbb{Z}/{p-1}$.  By Lemma \ref{nat}  the diagram 
\[\begin{tikzcd}
\ku \ar{d}[swap]{\delta} \ar{r}{\id} & \ku \ar{d} {\delta} \\
\ku \ar{r}{\id}  & \ku. 
\end{tikzcd}\]
induces a morphism of Brun spectral sequences $\delta^*_{*,*}$ converging to 
$\delta_*\colon V(1)_*^S\THH(ku) \to V(1)_*^S\THH(ku)$.  By linear algebra every finite-dimensional vector space over $\mathbb{F}_p$ admiting an automorphism $f$ with $f^{p-1} =\id$ splits into eigenspaces with respect to $f$. 
We get that the Brun spectral sequence splits into a direct sum of $p-1$ spectral sequence (corresponding to the $p-1$ possible eigenvalues). Furthermore, $V(1)_*^S\THH(\ku)$ splits into eigenspaces and the subspectral  sequence corresponding to an eigenvalue $\beta$ converges to  the eigenspace corresponding to $\beta$.  

For the class  $u  \in V(0)^S_*\ku$ we have $\delta_*(u) = \alpha u$ for a generator $\alpha$ of $\mathbb{F}_p^*$ \cite[p. 1267]{Au}.  It follows that the same is true for the classes $u \in V(1)_*^S\ku$ and $u \in (H\mathbb{F}_p)_*^S\ku$.  We say that a class has $\delta$-weight $i \in \mathbb{Z}/{p-1}$ if it lies in the eigenspace corresponding to $\alpha^i$.

The class $\sigma u \in E^2_{3,0}$ has $\delta$-weight $1$, because $u \in (H\mathbb{F}_p)_*^S\ku$ has $\delta$-weight $1$ and because the map $\sigma_*: (H\mathbb{F}_p)_*\ku \to (H\mathbb{F}_p)_{*+1}\THH(\ku)$ is natural.
We claim that the classes $\lambda_1, \mu_1 \in E^2_{*,*}$ have $\delta$-weight $0$:  
 Note that $\cup \mathbb{F}_{l^{p^i}}$ is the subfield of $\cup \mathbb{F}_{l^{p^i(p-1)}}$ fixed under $\delta$.   It follows that $\delta \colon \ku \to \ku$ restricts to the identity map on the Adams summand $\ell$. We consider the $\ell^e$-module structure on $H\mathbb{F}_p$ given by the morphism of commutative $S$-algebras 
\[\begin{tikzcd}
\ell \ar{r} & \ku \ar{r} & H\mathbb{Z}_p \ar{r} & H\mathbb{F}_p. 
\end{tikzcd}\]
We then have commutative diagram in $\mathscr{D}_S$
\[\begin{tikzcd}
H\mathbb{F}_p \wedge_{\ku^e}^L \ku \ar{r}{\cong} &  H\mathbb{F}_p \wedge_{\ku^e}^L \ku  \\
H\mathbb{F}_p \wedge_{\ell^e}^L \ell \ar{r} {\id} \ar{u} & H\mathbb{F}_p \wedge_{\ell^e}^L \ell \ar{u},
\end{tikzcd}\]
where the upper horizontal map is the one inducing $\delta^2_{*,*}$ on the second tensor factor in 
\[E^2_{*,*} = V(1)_*^S\ku \otimes \THH_*(\ku; H\mathbb{F}_p). \]
 It thus suffices to show that $\lambda_1$ and $\mu_1^p$ are in the image of $\pi_*(H\mathbb{F}_p \wedge_{\ell^e}^L \ell) \to \pi_*(H\mathbb{F}_p \wedge_{\ku^e}^L \ku)$. 
The diagram 
\[\begin{tikzcd}
\ku \ar{r} & H\mathbb{Z}_p \\
\ell \ar{r} \ar{u} & H\mathbb{Z}_p \ar{u}[swap]{\id} 
\end{tikzcd}\]
factors as 
\[\begin{tikzcd}
\ku \ar[tail]{r} & \hat{H}\mathbb{Z}_p \ar[two heads]{r}{\sim} & H\mathbb{Z}_p \\
\ell \ar{u} \ar[tail]{r} & \tilde{H}\mathbb{Z}_p \ar[two heads]{r}{\sim} \ar{u} & H\mathbb{Z}_p \ar{u}[swap]{\id}.
\end{tikzcd}\]
The map  $\pi_*(H\mathbb{F}_p \wedge_{\ell^e}^L \ell) \to \pi_*(H\mathbb{F}_p \wedge_{\ku^e}^L \ku)$ identifies with 
\[V(0)^S_*\THH(\ell; \tilde{H}\mathbb{Z}_p)  \to V(0)_*^S\THH(\ku; \hat{H}\mathbb{Z}_p)\]
and we have a morhism of Brun spectral sequences converging to this map. 
The $E^2$-page of the first Brun spectral sequence is 
\[E^2_{*,*}  = E(\sigma v_1) \otimes E(\lambda_1') \otimes P(\mu_1'), \]
where $|\sigma v_1| = (0, 2p-1)$, $|\lambda_1'| = (2p-1, 0)$ and $|\mu_1'| = (2p, 0)$. By \cite[Theorem X.2.4, Proposition IX.2.8, Theorem VII. 6.5, Theorem VII. 6.7]{EKMM} commutative diagrams of cofibrant commutative $S$-algebras, where the vertical maps are weak equivalences, induces a weak equivalence in $\THH(-;-)$. Therefore, the map  of spectral sequence has to send $\lambda_1'$ to $\lambda_1$ and $\mu_1'$ to $\mu_1$ up to a unit. 
Since $\lambda_1'$ and $\mu_1'^p$ have to survive to the $E^{\infty}$-page and since $V(0)_{2p^2}^S\THH(\ku, \hat{H}\mathbb{Z}_p)$  and  $V(0)^S_{2p-1}\THH(\ku, \hat{H}\mathbb{Z}_p)$ are one dimensional this proves that $\lambda_1$ and $\mu_1^p$ are in the image of $\pi_*(H\mathbb{F}_p \wedge_{\ell^e}^L \ell) \to \pi_*(H\mathbb{F}_p \wedge_{\ku^e}^L \ku)$. 

We get that the classes $u$, $\lambda_1$, $\mu_2$, $a_i$ for $i = 0 \dots, p-1$ and $b_i$ for $i = 1, \dots, p-1$ in the $E^{\infty}$-page of the Brun spectral sequence converging to $V(1)_*^S\THH(\ku)$ have the $\delta$-weights $1$, $0$, $0$, $1$ and $1$ respectively.  
One easily checks that these classes have unique representatives in $V(1)_*^S\THH(\ku)$  with the same $\delta$-weight.  We use the same names for the representatives.  

The relations  $u^{p-1}  = 0$, $u^{p-2} a_i = 0$ for $i = 0, \dots, p-2$, $u^{p-2}b_i = 0$ for $i = 1, \dots, p-1$,  $b_i b_j = u b_{i+j}$  for $i+j \leq p-1$ and $b_i b_j = u b_{i+j-p} \mu_2$ for $i+j \geq p$ hold in $V(1)_*^S\THH(\ku)$, because they hold in the $E^{\infty}$-page and because the $E^{\infty}$-page is zero in lower filtration in the respective total degrees.

We claim  that the relations $a_ib_j = u a_{i+j}$ for $i+j \leq p-1$ and $a_i b_j = u a_{i+j-p} \mu_2$ for $i+j \geq p$ hold in $V(1)_*^S\THH(\ku)$. For this note that the only classes with the same total degree and lower filtration degree in the $E^{\infty}$-page are 
\begin{eqnarray*}
\lambda_1 u^2 b_{i+j-1}, & \text{if} & 1 \leq i+j \leq p; \\
\lambda_1 u^2  b_{i+j-1-p} \mu_2,  & \text{if} &  p+1 \leq i+j \leq 2p-2;
\end{eqnarray*}
where we again use the convention $b_0 = u$. 
Since these classes have $\delta$-weight $3$, whereas $a_ib_j$ has $\delta$-weight $2$, this proves the claim. 
We claim that $a_ia_j = 0$ in $V(1)_*^ S\THH(\ku)$. For this note that 
the only classes with the same total degree and lower filtration degree in the $E^{\infty}$-page are
\begin{eqnarray*}
u^2 b_{i+j},  & \text{if} & 0 \leq i+j \leq p-1; \\
u^2 b_{i+j-p} \mu_2, & \text{if}  &  p  \leq i+j \leq  2p-2; \\ 
\lambda_1 u^2 a_{i+j-1} , & \text{if} & 1 \leq i + j \leq p; \\
\lambda_1 \mu_2 u^2 a_{i+j-1-p}, & \text{if} &  p+1 \leq i+j \leq 2p-2. 
\end{eqnarray*}
These classes have $\delta$-weight $3$, whereas $a_i a_j$ has $\delta$-weight $2$.  
\end{proof}

\printbibliography
\end{document}